\documentclass[11pt,reqno]{amsart}
\input{glyphtounicode}
\usepackage[T1]{fontenc}
\usepackage[utf8]{inputenc}
\usepackage{lmodern}
\usepackage[a4paper,margin=27mm]{geometry}
\usepackage{amsmath,amssymb,amsthm,mathtools}
\usepackage{microtype,booktabs,array,enumitem}
\usepackage{needspace}
\usepackage{xcolor}
\definecolor{linkblue}{RGB}{25,60,105}
\usepackage[colorlinks=true,linkcolor=linkblue,citecolor=linkblue,urlcolor=linkblue]{hyperref}
\usepackage{bookmark}
\usepackage[nameinlink,noabbrev]{cleveref}
\crefname{theorem}{Theorem}{Theorems}
\crefname{lemma}{Lemma}{Lemmas}
\crefname{proposition}{Proposition}{Propositions}
\crefname{corollary}{Corollary}{Corollaries}
\crefname{definition}{Definition}{Definitions}
\crefname{section}{Section}{Sections}
\crefname{subsection}{Section}{Sections}
\crefname{problem}{Problem}{Problems}
\crefname{question}{Question}{Questions}
\allowdisplaybreaks[2]
\setlist{itemsep=3pt,topsep=5pt}
\numberwithin{equation}{section}
\newtheorem{theorem}{Theorem}[section]
\newtheorem{lemma}[theorem]{Lemma}
\newtheorem{proposition}[theorem]{Proposition}
\newtheorem{corollary}[theorem]{Corollary}
\theoremstyle{definition}
\newtheorem{definition}[theorem]{Definition}
\newtheorem{problem}[theorem]{Problem}
\newtheorem{question}[theorem]{Question}
\theoremstyle{remark}

\title[Exact and extremal chromatic villainy]{Exact and extremal values of chromatic villainy}
\author[Jiang, Jiang, Li, Wang, and Zhang]{Jiang Jiang, Yuhan Jiang, Yiyang Li, Lujia Wang, and Aitao Zhang}
\address{Hainan Bielefeld University of Applied Sciences, Danzhou, China}
\email{lujia.wang@hibiuh.edu.cn}
\email{\{x20250083,x20250116,x20250157,x20250275\}@stu.hibiuh.edu.cn}
\hypersetup{pdftitle={Exact and extremal values of chromatic villainy},
pdfauthor={Jiang Jiang, Yuhan Jiang, Yiyang Li, Lujia Wang, Aitao Zhang}}
\date{7 October 2026}
\begin{document}
\begin{abstract}
Chromatic villainy measures the largest number of vertices that must
change color to repair a rearrangement of an optimal coloring while
preserving every color multiplicity. Weak villainy allows different
multiplicities in the repair. On the odd cycle $C_{2r+1}$, we determine
the largest repair distance from arbitrary color assignments with a
prescribed color-count vector (profile) $\mathbf m$ to proper colorings
with any realizable target profile $\mathbf p$: it equals
$2r+1-\sum_X\min\{p_X,\lceil m_X/2\rceil\}$.
In particular, every feasible profile $(a,b,c)$ has ordinary and weak
repair radius $\lfloor a/2\rfloor+\lfloor b/2\rfloor+\lfloor c/2\rfloor$,
and $B(C_{2r+1})=r$ for $r\ge2$, answering a question of Clark et al.
For arbitrary complete multipartite graphs, we give an arithmetic
formula for ordinary villainy with at most $k^2$ candidates and an exact
finite formula for weak villainy, where $k$ is the number of parts.
Both values are fixed-parameter tractable in $k$ for binary-encoded
part sizes. We also determine the maximum of each parameter among all
$n$-vertex $k$-chromatic graphs. For $k\ge3$, the ordinary maximum is
$\min\{n,4(n-k)\}$, except when $n=4(n-k)-1$ or
$(n,k)\in\{(5,3),(7,3)\}$, when it is $n-1$.
For $k\ge2$, the weak maximum is
$n-\max\{\lceil n/k\rceil,2k-n\}-\varepsilon$, where
$\varepsilon\in\{0,1\}$ is determined explicitly by $n$ and $k$.
Both extrema are attained by complete multipartite graphs, and a weak
extremizer needs at most three distinct part sizes.
\end{abstract}
\maketitle
\section{Introduction}

A rearrangement of the colors of a proper coloring preserves the number
of occurrences of each color but may destroy properness. Chromatic
villainy measures the greatest distance from such a rearrangement to a
proper coloring with the same color multiplicities. More precisely, if
$G$ has chromatic number $k$, its villainy $B(G)$ is the maximum, over
all proper $k$-colorings and all rearrangements of their colors, of the
minimum number of vertices that must change color in a count-preserving
repair. Its weak villainy $B_{\mathrm w}(G)$ allows the repaired
coloring to have different multiplicities, while using the same palette
of $k$ colors. Both distances count changed vertices. They concern the
initial and final colorings, with no prescribed sequence of intermediate
recolorings.

Clark et al.~\cite{clark2006} introduced these parameters and determined
them for several basic graph classes, including connected bipartite
graphs. For odd cycles they proved
$B_{\mathrm w}(C_{2r+1})=r$ and $B(C_{2r+1})\ge r$ for $r\ge2$,
and asked whether the latter bound is always attained. They also raised
questions about complete multipartite graphs and the largest possible
villainy at fixed order and chromatic number. These questions appear in
West's collection~\cite{west2006} and in the account of Jahanbekam and
Lin~\cite{jahanbekam2021}, who characterized the graphs with villainy
two. Raleigh's thesis~\cite{raleigh2018} obtained further exact values
and bounds for complete multipartite graphs. It also treated two odd-cycle profiles
\cite[Theorems~7 and~8]{raleigh2018}.
The general odd-cycle question remained open in both that thesis and
the later paper of Jahanbekam and Lin.

We approach these problems through the \emph{profile} of a coloring (\cref{def:profile}),
the vector of its labeled color multiplicities. For a feasible profile
$\mathbf m$, let $R_G(\mathbf m)$ be the largest, over all color assignments of
profile $\mathbf m$, of the distance to a proper coloring of that same
profile. Then $B(G)=\max_{\mathbf m}R_G(\mathbf m)$, with the maximum
over feasible profiles. This viewpoint keeps track of the restriction
that distinguishes ordinary villainy from weak villainy, and allows
the admissible source and target profiles to be varied separately.
On cycles, this leads to exact agreement guarantees at prescribed
profiles. On complete multipartite graphs, the same count constraints
give an assignment problem. Completing a graph relative to an optimal
coloring then extends the multipartite description to the largest
possible values over all graphs of given order and chromatic number.

Our first result determines the full directed repair radius on an odd
cycle. For any source profile $\mathbf m$ and any feasible target
profile $\mathbf p$, \cref{odd:directed} gives
\[
 R_{C_{2r+1}}(\mathbf m\to\mathbf p)
 =2r+1-\sum_X\min\{p_X,\lceil m_X/2\rceil\}.
\]
The source profile need not itself be feasible. In fact,
\cref{odd:capped} guarantees at least
$\min\{p_X,\lceil m_X/2\rceil\}$ agreements separately for each
color $X$. A source word with each color confined to one consecutive
block gives the radius for every
target profile simultaneously, so \cref{odd:directed} also treats any
nonempty family of allowed target profiles.

Taking the same feasible profile $\mathbf m=(a,b,c)$ at source and
target yields
\[
 R_{C_{2r+1}}(\mathbf m)
 =\left\lfloor\frac a2\right\rfloor
  +\left\lfloor\frac b2\right\rfloor
  +\left\lfloor\frac c2\right\rfloor,
\]
and hence $B(C_{2r+1})=r$ for $r\ge2$, as stated in
\cref{odd:main}. The matching obstruction and the floor sum already
occur in the weak-repair argument of
Clark et al.~\cite[Proposition~3.5]{clark2006}; the upper bound here
preserves every labeled multiplicity. The ordinary and weak radii
therefore agree at every feasible source profile. The proof constructs
auxiliary targets and interpolates for proper sources, then uses
induction for arbitrary sources.

For a complete multipartite graph $K_{p_1,\ldots,p_k}$, a disturbance
is described by a nonnegative integer matrix whose row and column
sums are $\mathbf p$. A weak repair chooses a maximum-weight
assignment. An ordinary repair restricts the assignment to parts of
equal size. The relevant groups of parts are therefore the
\emph{size classes}, each consisting of all parts of one size.
The maximal-value case is known: $B(G)=|V(G)|$ exactly when no size
class contains more than half the vertices. Raleigh proved the
sufficient direction~\cite[Theorem~1]{raleigh2018}; the equivalence
also follows from Cameron and Wanless's covering-radius theorem for
permutation groups~\cite[Theorem~14]{cameronwanless2005}.
Call a size class \emph{dominant} if it contains more than half the
vertices. An additional dominant-class regime was determined by
Raleigh~\cite[Theorem~4]{raleigh2018}.

\cref{thm:mp:ordinary} gives a general arithmetic formula,
including the remaining dominant-class cases, with at most $k^2$
candidates. For weak villainy, \cref{thm:mp:weak} expresses
the answer as a finite minimum over ordered integral assignment
potentials. A strengthened exchange bounds their coordinate gaps
by differences of part sizes, up to one. Thus the explicit search
depends on the size spread rather than a common baseline.
\Cref{cor:mp:computation} also gives fixed-parameter tractability
in $k$ for binary-encoded sizes through a fixed-dimensional integer
program: the running time is $f(k)L^c$, where $L$ is the input length,
$f$ is computable, and $c$ is an absolute constant.
When all part sizes are equal, the assignment problem
is the integer doubly stochastic problem treated by Dinitz, Hartman,
and Soprunova~\cite{dinitz2012}; Gajula, Soprunov, and
Soprunova~\cite{gajula2015} studied its rectangular extension.
The formulas here permit arbitrary positive part sizes. Integral
assignment duality and transportation cuts provide a common framework
for the two repair rules, although their resulting optimizations
are different.

We also determine both extremal values
\[
 \mathcal M(n,k)=\max_{\substack{|V(G)|=n\\\chi(G)=k}}B(G),
 \qquad
 \mathcal M_{\mathrm w}(n,k)
 =\max_{\substack{|V(G)|=n\\\chi(G)=k}}B_{\mathrm w}(G).
\]
Both maxima are attained by complete multipartite graphs. The reduction
completes a graph relative to a proper coloring whose profile supplies
an extremal disturbance; the disturbance remains admissible after
completion. For $k\ge3$, the ordinary extremum is
$\min\{n,4(n-k)\}$ apart from the explicitly listed exceptions in
\cref{ext:ordinary}. For weak villainy,
\cref{ext:weak} gives, for $k\ge2$,
$\mathcal M_{\mathrm w}(n,k)=n-L-\varepsilon$, where
$L=\max\{\lceil n/k\rceil,2k-n\}$ and the correction
$\varepsilon\in\{0,1\}$ is determined by a parity obstruction.
The obstruction comes from balance of the integer disturbance matrix:
its corresponding row and column sums agree. Explicit balanced
matrices show that this is the only correction required; attaining
weak extremizers need at most three distinct part sizes.

The majority-retention theorem is related to coloring with requests
\cite{dvorak2019}: each vertex requests one color, and one seeks a
proper coloring satisfying many requests. Quantitative request-satisfaction
bounds and structural aspects of flexible list coloring are studied
in~\cite{Kaul2024,Bowdoin2025}.
Here the quotas are separate for each requested color, and the target
multiplicities are prescribed simultaneously. Prescribed color-class
sizes also arise in the sampling results of Kuchukova, Perkins, and
Povill~\cite{Kuchukova2026} and the recent existence theorem of
Birken~\cite{Birken2026}.

\cref{sec:preliminaries} gives the profile formulation and the
transportation criterion. \cref{sec:oddcycles} proves the
odd-cycle results. \cref{sec:mp,sec:extremal}
treat complete multipartite graphs and the extrema at fixed order
and chromatic number. We conclude with questions about extremal
structure, arithmetic complexity, and prescribed-profile retention.

\section{Color profiles and a transportation criterion}\label{sec:preliminaries}

All graphs are finite and simple. Unless otherwise stated, a graph $G$ is
nonempty, has order $n$, and has chromatic number $k$. We use a fixed set
$\mathcal C$ of $k$ labeled colors. For a positive integer $s$, write
$[s]=\{1,\ldots,s\}$.

\begin{definition}[Words and profiles]\label{def:profile}
A \emph{word} on $G$ is a map
$w:V(G)\to\mathcal C$, with no properness assumption. Its \emph{profile}
is the vector $\mathbf m(w)=(m_X(w))_{X\in\mathcal C}$, where
$m_X(w)=|w^{-1}(X)|$. Bold letters denote vectors. Inequalities between
vectors are coordinatewise, $\mathbf e_X$ is a coordinate unit vector,
and $\mathbf1$ is an all-one vector of the required dimension.

Write $\mathcal P(G)$ for the proper colorings with palette $\mathcal C$.
Since $|\mathcal C|=\chi(G)$, each such coloring uses every color. A profile
$\mathbf m$ is \emph{feasible} if it occurs in $\mathcal P(G)$.
\end{definition}

For a feasible profile, put
\[
 \mathcal W_{\mathbf m}(G)=
 \{w:V(G)\to\mathcal C:\mathbf m(w)=\mathbf m\},\qquad
 \mathcal P_{\mathbf m}(G)=\mathcal W_{\mathbf m}(G)\cap\mathcal P(G).
\]
We omit $G$ from $\mathcal W_{\mathbf m}(G)$ when the vertex set is clear.
Every word in $\mathcal W_{\mathbf m}(G)$ is a rearrangement of every
proper coloring with that profile; we call it a \emph{legal disturbance}.
More generally, we call the given word the \emph{source} and a proper
coloring used to repair it the \emph{target}. An \emph{agreement} is a
vertex on which source and target have the same color.
The Hamming distance and the number of agreements are
\[
 d_{\mathrm H}(w,y)=|\{v:w(v)\ne y(v)\}|,
 \qquad \operatorname{agr}(w,y)=n-d_{\mathrm H}(w,y).
\]
The cost therefore counts changed vertices; exchanging two different
colors on two vertices has cost two.

\begin{definition}[Repair radii and villainy]\label{def:repair-radii}
The repair radius of a profile and its weak counterpart are
\begin{align*}
 R_G(\mathbf m)&=
 \max_{w\in\mathcal W_{\mathbf m}(G)}
 \min_{y\in\mathcal P_{\mathbf m}(G)}d_{\mathrm H}(w,y),\\
 R^{\mathrm w}_G(\mathbf m)&=
 \max_{w\in\mathcal W_{\mathbf m}(G)}
 \min_{y\in\mathcal P(G)}d_{\mathrm H}(w,y).
\end{align*}
With the terminology of Clark et al.~\cite{clark2006},
\[
 B(G)=\max_{\mathbf m\text{ feasible}}R_G(\mathbf m),\qquad
 B_{\mathrm w}(G)=\max_{\mathbf m\text{ feasible}}R^{\mathrm w}_G(\mathbf m).
\]
Both parameters are zero on the empty graph by convention.
\end{definition}

We also call $B(G)$ \emph{ordinary villainy}, to distinguish it from
weak villainy.

The inclusion
of the two target sets gives $B_{\mathrm w}(G)\le B(G)$. In particular,
the weak target may change the profile, but it still uses the same palette
of $\chi(G)$ colors. A simultaneous permutation of the color labels in
the source and target preserves distances and all profile conditions.

We use the following integral transportation criterion, a classical
consequence of the max-flow theorem~\cite{fordfulkerson1956}.
Its short proof is included because the precise capacity inequalities
are used both for gap allocations on cycles and for multipartite graphs.

For a vector $\mathbf a$ indexed by a finite set, write
$a(I)=\sum_{i\in I}a_i$.
\Needspace{6\baselineskip}
\begin{lemma}[Integral transportation criterion]\label{lem:transport}
Let $\mathbf a\in\mathbb Z_{\ge0}^{h}$ and
$\mathbf b\in\mathbb Z_{\ge0}^{j}$ have the same total $N$, and let
$c_{i\ell}$ be nonnegative integers. There is a nonnegative integer
matrix $M$ with row sums $\mathbf a$, column sums $\mathbf b$, and
$m_{i\ell}\le c_{i\ell}$ if and only if
\begin{equation}\label{eq:transport-cut}
 a(I)+b(J)-N\le\sum_{i\in I,\ell\in J}c_{i\ell}
 \qquad(I\subseteq[h],\ J\subseteq[j]).
\end{equation}
\end{lemma}
\begin{proof}
At most $N-b(J)$ units from rows in $I$ can lie outside the columns in
$J$. This proves necessity. For sufficiency, form a network whose arcs
from the source to row $i$, from row $i$ to column $\ell$, and from column
$\ell$ to the sink have capacities $a_i,c_{i\ell},b_\ell$, respectively.
A cut with source-side rows $I$ and sink-side columns $J$ has capacity
\[
 2N-a(I)-b(J)+\sum_{i\in I,\ell\in J}c_{i\ell}.
\]
The assumed inequalities make every cut have capacity at least $N$.
Starting with zero flow, augment along residual source--sink paths by
positive integers. The value increases at every step and is bounded by
$N$. If augmentation stops before value $N$, the vertices reachable from
the source determine a cut whose capacity is the current value, a
contradiction. An integral flow of value $N$ saturates all source and sink
arcs. Its row-to-column values give the required matrix.
\end{proof}

The same augmenting-path argument applies to every integral network used
below. In particular, when the row and column bounds are upper bounds
rather than prescribed sums, the largest possible integral total is the
minimum cut capacity.

\section{Odd cycles and prescribed color multiplicities}\label{sec:oddcycles}

Let $n=2r+1$, and use the three labeled colors
$\mathcal C=\{A,B,C\}$. We identify a coloring of $C_n$ with a
cyclic word $w$. Its profile is
$\mathbf m=(m_A,m_B,m_C)$, where
$m_X=|w^{-1}(X)|$. The profile is \emph{feasible} if some proper
coloring has these multiplicities.
We write $X^a$ for a consecutive block of $a$ copies of color $X$,
with the block empty when $a=0$.

\begin{theorem}\label{odd:main}
For every feasible profile $\mathbf m=(a,b,c)$ of $C_{2r+1}$,
\begin{equation}\label{odd:radius}
R_{C_{2r+1}}(\mathbf m)
=\left\lfloor\frac a2\right\rfloor
 +\left\lfloor\frac b2\right\rfloor
 +\left\lfloor\frac c2\right\rfloor
=\begin{cases}
r-1,&a,b,c\text{ are all odd},\\
r,&\text{exactly one of }a,b,c\text{ is odd}.
\end{cases}
\end{equation}
Consequently,
$B(C_{2r+1})=r$ for $r\ge2$, and $B(C_3)=0$.
\end{theorem}

Clark et al.~\cite[Proposition~3.5 and Corollary~3.6]{clark2006}
proved $B_{\mathrm w}(C_{2r+1})=r$ and $B(C_{2r+1})\ge r$ for $r\ge2$.
The floor sum in~\eqref{odd:radius} also occurs in their argument.
Raleigh~\cite[Theorems~7--8]{raleigh2018} proved upper bounds for the
profiles $(r,r,1)$ and $(r,r-1,2)$, respectively. Her first proof
already gives $2\lfloor r/2\rfloor$, and the second gives $r$;
these agree with~\eqref{odd:radius}. They are profile-restricted
statements, not the full odd-cycle formula.
The issue addressed here is the requirement that the repaired coloring
preserve each labeled color multiplicity. We prove a stronger assertion
in which the target multiplicities may be prescribed independently of
the source multiplicities.

\subsection{Profiles and agreement quotas}

\begin{lemma}\label{odd:feasible}
A nonnegative integer vector $\mathbf m$ is feasible on $C_{2r+1}$
if and only if
\begin{equation}\label{odd:feasibleeq}
\sum_{X\in\mathcal C}m_X=2r+1,
\qquad 1\le m_X\le r\quad(X\in\mathcal C).
\end{equation}
\end{lemma}
\begin{proof}
Each color class of a proper coloring is independent and therefore has
at most $r$ vertices. All three colors occur, since an odd cycle is not
bipartite. This proves necessity.

Conversely, suppose~\eqref{odd:feasibleeq} holds. Form a loopless
multigraph on $A,B,C$ in which the multiplicity of the edge $AB$ is
$m_A+m_B-m_C$, and define the other two multiplicities cyclically.
For example, the multiplicity of $AB$ is $2r+1-2m_C\ge1$.
Thus this multigraph is connected, has $2r+1$ edges, and has degree
$2m_X$ at $X$. It has an Euler circuit. Reading the successive vertices
of this circuit, with the final repeated vertex omitted, gives a cyclic
word in which consecutive letters differ and $X$ occurs $m_X$ times.
\end{proof}

For an arbitrary source word $w$, define its half-count requirements by
\[
\mathbf q=(q_A,q_B,q_C),
\qquad q_X=\left\lceil\frac{m_X}{2}\right\rceil.
\]
The source need not be proper, and some of its counts may be zero.
For a feasible target profile $\mathbf p$, the corresponding agreement
quotas are
\[
\mathbf t=(t_A,t_B,t_C),\qquad t_X=\min\{p_X,q_X\}.
\]
The target count caps a color's half-count requirement when fewer
vertices of that color are available in the target.

\begin{theorem}[Prescribed multiplicities and capped majority retention]
\label{odd:quota}\label{odd:capped}
Let $w$ be a word on $C_{2r+1}$, where $r\ge1$, with profile
$\mathbf m$. If an integer vector $\mathbf p$ satisfies
\begin{equation}\label{odd:target}
\sum_{X\in\mathcal C}p_X=2r+1,
\qquad 1\le p_X\le r
\quad(X\in\mathcal C),
\end{equation}
then there is a proper coloring $y$ of profile $\mathbf p$ such that
\begin{equation}\label{odd:agreements}
|\{v:w_v=y_v=X\}|\ge t_X=\min\{p_X,q_X\}
\quad(X\in\mathcal C).
\end{equation}
\end{theorem}

For distinct colors $X,Y$, we shall use the inequalities
\begin{equation}\label{odd:pair-quota}
q_X+q_Y\le r+1\le p_X+p_Y.
\end{equation}
The first follows from $m_X+m_Y\le2r+1$, and the second from the
upper bound $r$ on the third target count. In particular, if
$p_A<q_A$, both other target counts strictly exceed their half-count
requirements. We call a target coordinate $X$ \emph{deficient} if
$p_X<q_X$ and \emph{slack} if $p_X>q_X$.
At most one coordinate can be deficient.

We first deduce the profile-radius formula, assuming this theorem.

\begin{proof}[Proof of \cref{odd:main}]
For a feasible profile $\mathbf m$, set $\mathbf p=\mathbf m$ in
\cref{odd:quota}. Its hypotheses hold by
\cref{odd:feasible}, and then $t_X=q_X$ for every color.
Every source word of profile $\mathbf m$
has a proper repair of the same profile retaining at least
$\sum_X\lceil m_X/2\rceil$ vertices. Hence
\[
R_{C_{2r+1}}(\mathbf m)
\le\sum_{X\in\mathcal C}\left\lfloor\frac{m_X}{2}\right\rfloor.
\]
The matching lower bound is the three-block obstruction in the proof
of Clark et al.~\cite[Proposition~3.5]{clark2006}: for the source
$A^aB^bC^c$, every proper coloring differs on at least
$\lfloor a/2\rfloor+\lfloor b/2\rfloor+\lfloor c/2\rfloor$
vertices, even when target multiplicities are unrestricted.
It applies in particular to the smaller family of proper targets of
profile $\mathbf m$, proving equality.

If $o\in\{1,3\}$ is the number of odd entries of $\mathbf m$, then
\[
\sum_X\left\lfloor\frac{m_X}{2}\right\rfloor
=\frac{2r+1-o}{2}.
\]
This gives the two cases in~\eqref{odd:radius}. The value is always
at most $r$. For even $r\ge2$, the feasible profile $(r,r,1)$ attains
$r$; for odd $r\ge3$, the feasible profile $(r,r-1,2)$ does so.
When $r=1$, the sole feasible profile is $(1,1,1)$, and every
rearrangement is already proper.
\end{proof}

There are two parts to the proof of \cref{odd:quota}. We first treat proper sources by
constructing suitable auxiliary profiles and interpolating while keeping
a specified set of agreements fixed. We then remove an equal adjacent
pair from an improper source and apply induction. Two elementary
observations about alternating colorings will also be useful.

\Needspace{8\baselineskip}
\begin{lemma}\label{odd:binary}
Suppose a set whose vertices are colored $B$ or $C$ is partitioned
into $E$ and $O$, where $\bigl||E|-|O|\bigr|\le1$.
Of the two assignments
\[
(B\text{ on }E,\ C\text{ on }O)
\quad\text{and}\quad
(C\text{ on }E,\ B\text{ on }O),
\]
one retains at least the ceiling of half of the original vertices of
each color.
\end{lemma}
\begin{proof}
Write $b_E,b_O,c_E,c_O$ for the original color counts in the two
parts. The first assignment retains $b_E$ vertices of color $B$ and
$c_O$ vertices of color $C$; the second retains $b_O$ and $c_E$,
respectively. Their agreement differences are
\[
d_B=b_E-b_O,\qquad d_C=c_O-c_E.
\]
Since $d_B-d_C=|E|-|O|$, the integers $d_B,d_C$ cannot have opposite
strict signs. Consequently one assignment is at least as good as the
other for both colors. For each color, its two possible agreement
counts sum to its original multiplicity, so the better count is at
least the ceiling of half that multiplicity.
\end{proof}

To apply this alternating choice after deleting a vertex of the cycle,
we need a root that balances the original vertices of one color between
the two parity classes of the remaining path, namely the vertices in
its odd and even positions.

\Needspace{6\baselineskip}
\begin{lemma}\label{odd:root}
For every nonempty vertex set $S$ in an odd cycle, there is $s\in S$
such that the two parity classes of the path $C_n-s$ contain
numbers of vertices of $S\setminus\{s\}$ differing by at most one.
\end{lemma}
\begin{proof}
For $s\in S$, let $\delta(s)$ be the number of other vertices of $S$
at even clockwise distance from $s$, minus the number at odd distance.
If $|S|=1$, then $\delta(s)=0$ and there is nothing to prove.
Otherwise let $t$ be the next vertex of $S$ clockwise from $s$, at
distance $g$. No other vertex of $S$ lies between them, so shifting
the root from $s$ to $t$ gives
\begin{equation}\label{odd:rootrecurrence}
\delta(t)=
\begin{cases}
\delta(s)-2,&g\text{ even},\\
-\delta(s),&g\text{ odd}.
\end{cases}
\end{equation}
To see this, distances to all other vertices decrease by $g$.
The vertex $t$, at distance $g$ from the old root, is replaced in the
count by $s$, at distance $n-g$ from the new root. Since $n$ is odd,
these two distances have opposite parity, giving~\eqref{odd:rootrecurrence}.

Suppose $|\delta(s)|>1$ for every $s\in S$. Across an even gap,
subtracting two cannot change the sign without producing a value
of absolute value at most one. Across an odd gap the sign reverses.
The gaps around the cycle sum to odd $n$, so an odd number of them
are odd. After one complete tour the sign would therefore reverse,
although the root has returned to its starting point. This
contradiction proves the existence of $|\delta(s)|\le1$.
\end{proof}

\Cref{odd:root,odd:binary} settle the case in which one
target color occurs only once.

\begin{lemma}\label{odd:singleton}
\cref{odd:quota} holds if a target multiplicity is $1$.
\end{lemma}
The same-profile singleton case has the elementary alternating-path
proof in~\cite[Theorem~7]{raleigh2018}. We include the following argument
because the induction needs independent source and target profiles,
possibly with zero source counts, and a separate agreement quota for
each color.
\begin{proof}
Relabel the colors so that $p_C=1$. The sum and upper bounds
in~\eqref{odd:target} force $p_A=p_B=r$. If the source is
monochromatic, every proper coloring of this target profile has
exactly $p_X=t_X$ agreements in its sole source color $X$, so any
such coloring suffices. Otherwise every source count is at most
$2r$, and hence $q_A,q_B\le r$.

First suppose $m_C\ge1$. By \cref{odd:root}, choose an original
$C$ vertex $s$ such that the two parity classes of $C_n-s$ contain
numbers of other original $C$ vertices differing by at most one.
Keep $s$ colored $C$. Each parity class has $r$ vertices, so after
ignoring the other original $C$ positions, the remaining original
$A,B$ positions in the two classes also have sizes differing by at
most one. \cref{odd:binary}, with the colors renamed, therefore
selects an alternating $A,B$ coloring retaining at least $q_A,q_B$
agreements. Assign its alternating colors also to the ignored
positions. The root is the unique target $C$, so its incident edges
are proper, and its one agreement meets $t_C=1$.

Now suppose $m_C=0$. Of $m_A,m_B$, precisely one is even. Its value
is positive because the source is not monochromatic.
Choose a vertex of the positive even-count
color and assign it the unique target $C$. Decreasing a positive even
count by one leaves its half-ceiling unchanged. On the remaining path,
whose two parity classes have equal size, \cref{odd:binary} therefore
gives an alternating coloring
meeting the original $A,B$ quotas. The $C$ quota is zero.
\end{proof}

\subsection{Interpolation with fixed agreements}

We next describe the profiles of all proper extensions of a prescribed
partial coloring: these are proper colorings that agree with every
prescribed color. The unprescribed positions between consecutive
prescribed vertices form a \emph{gap}, which may be empty.
The fact that there are exactly three colors is
essential: the relevant allocation conditions reduce to one lower and
one upper bound for each coordinate.

\Needspace{6\baselineskip}
\begin{lemma}\label{odd:gap}
Consider a path with $\ell\ge0$ internal vertices and prescribed endpoint
colors $s,t$. If $\ell=0$, assume $s\ne t$. Let $e_X$ be the number of
endpoints colored $X$, counted with multiplicity. An integer vector
$\mathbf x$ is the profile of a proper filling of the internal vertices
if and only if
\begin{equation}\label{odd:gapeq}
\sum_{X\in\mathcal C}x_X=\ell,
\qquad
0\le x_X\le
\left\lfloor\frac{\ell+1-e_X}{2}\right\rfloor
\quad(X\in\mathcal C).
\end{equation}
\end{lemma}
\begin{proof}
Every internal occurrence of $X$ is incident with two edges, and
every endpoint occurrence with one. Properness prevents an edge from
having both ends colored $X$. Since the path has $\ell+1$ edges,
$2x_X+e_X\le\ell+1$, proving the upper bounds.

For sufficiency, form a loopless multigraph on the three colors. Give
the edge joining the other two colors, opposite $X$, multiplicity
\[
h_X=\ell+1-2x_X-e_X.
\]
The assumptions make these multiplicities nonnegative integers.
Their sum is $\ell+1$, and the degree at $X$ is $2x_X+e_X$.
A loopless multigraph on three vertices has at most one component
containing edges. Both prescribed endpoint colors belong to that
component. If $s\ne t$, precisely $s$ and $t$ have odd degree, so
there is an Euler trail from $s$ to $t$; if $s=t$, there is an Euler
circuit starting there. The sequence of colors along this trail gives
a proper path. Its number of internal occurrences of $X$ is
$(\deg(X)-e_X)/2=x_X$, as required.
\end{proof}

We use the following specialization of the integral transportation
criterion, \cref{lem:transport}. Suppose a gap $g$ has $\ell_g$
positions and can receive at most $u_{gX}$ occurrences of color $X$.
Assume all these quantities are nonnegative integers and
$\sum_Xu_{gX}\ge\ell_g$. For nonnegative integer demands $d_X$ of total
$M=\sum_g\ell_g$, a filling of all gaps with those color totals exists
if and only if
\begin{equation}\label{odd:allocation}
\sum_{X\in S}d_X
\le\sum_g\min\left\{\ell_g,\sum_{X\in S}u_{gX}\right\}
\quad(S\subseteq\mathcal C).
\end{equation}
Indeed, apply \cref{lem:transport} with rows indexed by the gaps,
row sums $\ell_g$, column sums $d_X$, and capacities $u_{gX}$.
For a fixed set $S$ of columns, its inequalities for all row sets $I$
are equivalent to
\[
\sum_{X\in S}d_X
\le \sum_{g\notin I}\ell_g+
    \sum_{g\in I}\sum_{X\in S}u_{gX}.
\]
Minimizing the right side independently for each row gives
exactly~\eqref{odd:allocation}.

We now combine the gap capacities with this allocation criterion to
describe every profile compatible with a fixed partial coloring.

\begin{lemma}[Fixed-set interpolation]\label{odd:interpolation}
Let $K\ne\varnothing$ be a set of vertices of $C_n$ with prescribed
colors admitting a proper extension. Suppose profiles $\mathbf x$ and
$\mathbf z$ occur among its proper extensions. Then an integer vector
$\mathbf p$ also occurs as the profile of a proper extension of $K$
whenever
\[
\sum_Xp_X=n,
\qquad
\min\{x_X,z_X\}\le p_X\le\max\{x_X,z_X\}
\quad(X\in\mathcal C).
\]
\end{lemma}
\begin{proof}
Let $k_X$ be the number of prescribed vertices colored $X$. The
vertices outside $K$ split into gaps between consecutive vertices of
$K$ in cyclic order. Include gaps of length zero. If $|K|=1$, the
same fixed vertex supplies both endpoints of the single gap.
Write $\ell_g$ for the number of internal vertices in gap $g$ and
$e_{gX}$ for the number of its endpoints colored $X$.

\cref{odd:gap} gives the capacities
\[
u_{gX}=\left\lfloor\frac{\ell_g+1-e_{gX}}2\right\rfloor.
\]
The assumed proper extension ensures that these capacities are
nonnegative and that $\sum_Xu_{gX}\ge\ell_g$ for every gap.
Moreover, $u_{gX}\le\ell_g$. Define
\begin{align*}
U_X&=k_X+\sum_gu_{gX},\\
L_X&=k_X+
\sum_g\max\left\{0,\ell_g-\sum_{Y\ne X}u_{gY}\right\}.
\end{align*}
We claim that the extension profiles are precisely
\begin{equation}\label{odd:box}
\left\{\mathbf p\in\mathbb Z^3:
\sum_Xp_X=n,\quad L_X\le p_X\le U_X
\quad(X\in\mathcal C)\right\}.
\end{equation}

An extension with profile $\mathbf p$ requires a total of
$d_X=p_X-k_X$ occurrences of $X$ in the gaps.
For a singleton $S=\{X\}$, condition~\eqref{odd:allocation} is
$p_X\le U_X$. For the two-color set
$S=\mathcal C\setminus\{X\}$, use $\sum_Xd_X=\sum_g\ell_g$ to rewrite
the same condition as
\[
d_X\ge\sum_g
\left(\ell_g-\min\left\{\ell_g,\sum_{Y\ne X}u_{gY}\right\}\right),
\]
which is $p_X\ge L_X$. The empty and full color sets impose no
additional restrictions. Also $L_X\ge k_X$, so these lower bounds
ensure $d_X\ge0$. Thus~\eqref{odd:allocation} supplies an integer
allocation exactly when~\eqref{odd:box} holds, and
\cref{odd:gap} realizes each allocated gap by a proper path.
This proves the claim.

Both $\mathbf x$ and $\mathbf z$ satisfy the coordinate intervals
in~\eqref{odd:box}. Every vector coordinatewise between them satisfies
the same intervals. If its total is $n$, it therefore has the required
proper extension.
\end{proof}

In our applications, $K$ will consist of selected agreements with the
source. Interpolation keeps these particular agreements fixed. It does
not require all agreements of either extension to be retained.

\subsection{Proper source colorings}

Throughout this subsection, $w$ is proper. Hence
$1\le m_X\le r$ and $q_X\ge1$ for every color. All target profiles
satisfy~\eqref{odd:target}. The first two target constructions assume
$\mathbf p\ge\mathbf q$; we then treat a deficient coordinate separately.

Call a target color \emph{saturated} if its multiplicity is $r$,
the largest feasible count. A target with $r$ copies of $A$ needs an independent $A$ class of size
$r$. \Cref{odd:fair,odd:saturated} choose that class with enough source $A$
vertices inside it and enough source $B,C$ vertices outside it, then
complete the target coloring.

\begin{lemma}\label{odd:fair}
There is an independent set $I$ of size $r$ such that
\begin{align*}
|I\cap w^{-1}(A)|&\ge q_A,\\
|I\cap w^{-1}(B)|&\le\left\lfloor\frac{m_B}{2}\right\rfloor,
&
|I\cap w^{-1}(C)|&\le\left\lfloor\frac{m_C}{2}\right\rfloor.
\end{align*}
Moreover, the unique edge of the subgraph induced by $V(C_n)\setminus I$
has original endpoint colors $B$ and $C$.
\end{lemma}
\begin{proof}
Index the vertices modulo $n$, and put
\[
I_j=\{j+2,j+4,\ldots,j+2r\}.
\]
Each $I_j$ is independent, and its complement induces precisely the
edge $\{j,j+1\}$. The sets $I_j,I_{j+1}$ are disjoint and together
contain every vertex except $j+1$.

Call $I_j$ admissible if its numbers of original $B$ and $C$ vertices
satisfy the two asserted upper bounds. An admissible set exists.
Otherwise, label each $I_j$ by a color whose bound it violates,
choosing either one if both are violated. Two consecutive sets cannot
receive the same label: being disjoint, they would then contain at
least $2(\lfloor m_X/2\rfloor+1)>m_X$ vertices of that color $X$.
The labels would give a proper two-coloring of the odd cycle of
indices, a contradiction.

Every admissible set contains at least
\[
r-\left\lfloor\frac{m_B}{2}\right\rfloor
 -\left\lfloor\frac{m_C}{2}\right\rfloor
\]
original $A$ vertices. Checking the parities using
$m_A+m_B+m_C=2r+1$, this number is at least $q_A$ unless
$m_A$ is odd and $m_B,m_C$ are even. In that exceptional case it is
$q_A-1$. If an admissible $I_j$ fails the $A$ quota, its three counts
must therefore be exactly
\[
\left(\frac{m_A-1}{2},\frac{m_B}{2},\frac{m_C}{2}\right).
\]
The adjacent sets $I_{j-1}$ and $I_{j+1}$ are disjoint from $I_j$
and omit, respectively, $j$ and $j+1$ from their union with it.
At least one of those omitted vertices is not originally $A$, because
$w$ is proper on the edge $\{j,j+1\}$. The corresponding adjacent
set contains $(m_A+1)/2=q_A$ original $A$ vertices. Its $B,C$
counts are at most $m_B/2,m_C/2$, respectively, so it is admissible.
Thus an admissible set meeting the $A$ quota exists in every case.

Choose an admissible $I_j$ maximizing its number of original $A$
vertices. If $w_j=A$, replace $I_j$ by $I_{j-2}$; this removes
$j-1$ and inserts $j$. Properness gives $w_{j-1}\ne A$, so the
$A$ count increases and neither of the other counts increases.
This contradicts maximality. Similarly, if $w_{j+1}=A$, replacing
$I_j$ by $I_{j+2}$ removes $j+2$ and inserts $j+1$, with the same
contradiction. Thus neither endpoint of the complementary edge is
originally $A$. Their original colors are distinct, so they are $B,C$.
\end{proof}

\begin{lemma}\label{odd:saturated}
Suppose $\mathbf p\ge\mathbf q$. If a target multiplicity equals $r$,
a proper source has a proper
target satisfying all agreement quotas.
\end{lemma}
\begin{proof}
Suppose $p_A=r$, and color every vertex of the set $I$ in
\cref{odd:fair} with $A$. This retains at least $q_A$ original
$A$ vertices. The complement of $I$ contains at least
$m_X-\lfloor m_X/2\rfloor=q_X$ original vertices of color $X$,
for each $X\in\{B,C\}$.

In that complement, select $q_B$ original $B$ vertices and $q_C$
original $C$ vertices to retain. Include the corresponding endpoints
of its unique edge; this is possible because $q_B,q_C\ge1$.
There remain $p_B-q_B$ additional $B$ assignments and $p_C-q_C$
additional $C$ assignments. These numbers are nonnegative and their
sum equals the number of unassigned vertices, so fill those vertices
accordingly. The sole edge internal to the complement is proper, and
every edge from the complement to $I$ joins $B$ or $C$ to $A$.
\end{proof}

For a smaller increase in one target color, we use a saturated target
as an endpoint and interpolate back toward the source profile.

\begin{lemma}\label{odd:oneincrease}
Suppose $\mathbf p\ge\mathbf q$ and $p_X>m_X$ for at most one color $X$.
A proper source then
has a proper target satisfying all agreement quotas.
\end{lemma}
\begin{proof}
If no coordinate increases, the equality of the totals forces
$\mathbf p=\mathbf m$, and $w$ itself suffices. Otherwise suppose
$A$ is the sole increasing color. We construct a profile
$\mathbf P=(r,P_B,P_C)$ with
\[
q_B\le P_B\le p_B,\qquad
q_C\le P_C\le p_C,\qquad P_B+P_C=r+1.
\]
Choose an integer $P_B$ in the interval
\[
\max\{q_B,r+1-p_C\}
\le P_B\le
\min\{p_B,r+1-q_C\},
\qquad P_C=r+1-P_B.
\]
This interval is nonempty. Indeed, $q_B\le p_B$ and $q_C\le p_C$
hold by assumption, while
\[
p_B+p_C=2r+1-p_A\ge r+1,
\qquad q_B+q_C\le r+1.
\]
The last inequality also follows from~\eqref{odd:pair-quota}.
These four inequalities compare each lower endpoint with
each upper endpoint. The resulting $\mathbf P$ is feasible and
dominates all quotas: its $B,C$ entries lie below the corresponding
entries of $\mathbf p$, and $q_A\le m_A\le r=P_A$.

\cref{odd:saturated} gives a proper coloring $z$ of profile
$\mathbf P$ meeting the quotas. Select exactly $q_X$ agreements
between $w$ and $z$ of each color $X$, and let $K$ be their union,
with the common colors prescribed there. Both $w$ and $z$ are proper
extensions of $K$. Furthermore,
\[
m_A<p_A\le P_A,
\qquad P_B\le p_B\le m_B,
\qquad P_C\le p_C\le m_C.
\]
Thus $\mathbf p$ lies coordinatewise between $\mathbf m$ and
$\mathbf P$. \cref{odd:interpolation} supplies a proper
extension of $K$ with profile $\mathbf p$. Its agreements on $K$
meet every quota.
\end{proof}

We also need a direct construction when the target has too few copies
of one color to meet its full half-count requirement.

\begin{lemma}\label{odd:lower}
If $p_A\le q_A$, a proper source has a proper target satisfying all
agreement quotas.
\end{lemma}
\begin{proof}
By~\eqref{odd:pair-quota}, $p_B\ge q_B$ and $p_C\ge q_C$.
The required quotas are therefore $t_A=p_A$, $t_B=q_B$, and $t_C=q_C$.
Set $a=m_A$ and $h=p_A-1$.
Apply \cref{odd:root} to the original $A$ class. Keep the chosen
root $s$ colored $A$. The two parity classes of $C_n-s$ each have
$r$ vertices, and their original $A$ counts differ by at most one.
Each therefore contains at least
$\lfloor(a-1)/2\rfloor=q_A-1\ge h$ original $A$ positions.

Temporarily color the path alternately $B,C$. To choose between the
two alternating colorings, ignore the original $A$ positions when counting
agreements. The two sets of remaining original $B,C$ vertices have
sizes differing by at most one, because the full parity classes have
equal size and their original $A$ counts differ by at most one.
\cref{odd:binary} therefore gives a choice retaining at
least $q_B$ original $B$ vertices and $q_C$ original $C$ vertices.
Together with the root, this is a proper coloring of profile $(1,r,r)$.

Now recolor to $A$ precisely $r-p_B$ original $A$ positions currently
colored $B$, and $r-p_C$ original $A$ positions currently colored $C$.
The required numbers are nonnegative, and
\[
(r-p_B)+(r-p_C)
=2r-(2r+1-p_A)=p_A-1=h.
\]
Each number is therefore at most $h$, and each parity class contains
enough original $A$ positions. Including the root, all newly
$A$-colored vertices lie in the original $A$ class, which is
independent. Thus the recoloring remains proper. Its profile is
$\mathbf p$, and it retains exactly $p_A=t_A$ of the original $A$ vertices.
None of the selected $B,C$ agreements has been changed.
\end{proof}

The preceding lemmas leave only targets that increase two color counts
from a proper source. The boundary case $P_A=q_A$ of \cref{odd:lower}
provides the second endpoint for interpolation.

\begin{proposition}\label{odd:proper}
\cref{odd:quota} holds whenever the source is proper.
\end{proposition}
\begin{proof}
If some coordinate is deficient, relabel it $A$ and apply
\cref{odd:lower}. We may therefore assume $\mathbf p\ge\mathbf q$,
so the agreement quotas are the half-count requirements $q_X$.
\cref{odd:oneincrease} treats every target with at most one
increasing coordinate. In the remaining case, precisely two
coordinates increase and the third decreases. Relabel so that
\[
p_B>m_B,\qquad p_C>m_C,\qquad p_A<m_A.
\]
We first construct a target on the boundary $P_A=q_A$, further in
the same coordinate directions than $\mathbf p$.
Set $t=p_A-q_A\ge0$. Distribute these $t$ units between the two
increasing colors by choosing an integer $u$ satisfying
\[
0\le u\le r-p_B,\qquad 0\le t-u\le r-p_C.
\]
Such a choice exists because the total available capacity is
\[
(r-p_B)+(r-p_C)=p_A-1\ge p_A-q_A=t.
\]
For example, one may take any integer in
$[\max\{0,t-(r-p_C)\},\min\{t,r-p_B\}]$.

The vector
\[
\mathbf P=(q_A,p_B+u,p_C+t-u)
\]
has total $2r+1$, every coordinate lies between $1$ and $r$, and
every coordinate dominates the corresponding quota.
\cref{odd:lower} gives a proper coloring $z$ of this profile
meeting all quotas. As before, choose exactly $q_X$ agreements of
each color between $w$ and $z$ and prescribe their common colors
on the resulting nonempty set $K$.

Both $w$ and $z$ are proper extensions of $K$. The inequalities
\[
P_A=q_A\le p_A<m_A,
\qquad m_B<p_B\le P_B,
\qquad m_C<p_C\le P_C
\]
place $\mathbf p$ coordinatewise between their profiles.
\cref{odd:interpolation} gives a proper extension with profile
$\mathbf p$, and the fixed agreements on $K$ supply the quotas.
\end{proof}

\subsection{Induction for arbitrary sources}

For an improper source, an equal adjacent pair supplies the two
vertices deleted in the induction. The target counts are reduced in
two colors so that reinserting the pair restores both its profile and
its agreement quotas.

\begin{proof}[Proof of \cref{odd:quota}]
We induct on the odd cycle length $n=2r+1$.
For $n=3$, the target profile must be $(1,1,1)$ and every source
color that occurs has quota one. Retain one vertex of each such
color, then assign the missing colors to the remaining vertices.
The resulting coloring uses all three colors once, hence is proper
and meets every quota, including for a monochromatic source.

Suppose $n\ge5$. \cref{odd:singleton} handles targets having
a coordinate equal to one, and \cref{odd:proper} handles
proper sources. We may therefore assume that the source contains an
equal adjacent pair and that
\[
2\le p_X\le r\quad(X\in\mathcal C).
\]
At most one target count equals $r$, since two such counts would
force the third to be one. In particular, the remaining case has
$r\ge3$.

We first show that, for every source color $A$ occurring on an equal
adjacent pair, some different color $B$ has
\begin{equation}\label{odd:slack}
p_B>q_B.
\end{equation}
Indeed, the repeated pair gives $m_A\ge2$, so the other two
source counts have total at most $2r-1$. Consequently
\[
q_B+q_C
\le\left\lceil\frac{m_B+m_C+1}{2}\right\rceil\le r,
\qquad
p_B+p_C=2r+1-p_A\ge r+1.
\]
At least one of those two target counts strictly exceeds its
half-count requirement; label that color $B$.
This proves~\eqref{odd:slack}.

Choose an equal adjacent pair $AA$ and a different slack color $B$
so that any target color of multiplicity $r$ is one of $A,B$.
If no target multiplicity is $r$, there is nothing further to arrange.
If the unique such color is $X$ and the source contains an adjacent
$XX$, take $A=X$ and use~\eqref{odd:slack} to choose $B$.
If the source contains no $XX$, its $X$ vertices form an independent
set, so
\[
m_X\le r,
\qquad q_X\le\left\lceil\frac r2\right\rceil<r=p_X.
\]
Choose any equal adjacent pair $AA$ and take $B=X$. This also
has the required properties.

Delete the two vertices of the selected pair, and join their outside
neighbors to obtain $C_{n-2}$. Its source profile is
\[
\mathbf m'=\mathbf m-2\mathbf e_A.
\]
Prescribe the smaller target profile
\[
\mathbf p'=\mathbf p-\mathbf e_A-\mathbf e_B,
\]
where $\mathbf e_X$ is the coordinate unit vector for $X$.
The new half-count requirements satisfy
\[
\mathbf q'=\mathbf q-\mathbf e_A,
\]
because $\lceil(m_A-2)/2\rceil=\lceil m_A/2\rceil-1$.
For coordinate $A$, the target count and half-count requirement
both decrease by one. Thus its capped quota also decreases by one:
\[
t'_A=\min\{p_A-1,q_A-1\}=t_A-1.
\]
For coordinate $B$, the strict integer inequality~\eqref{odd:slack}
gives $p_B-1\ge q_B$, so $t'_B=q_B=t_B$. The third capped quota
is unchanged. Hence $\mathbf t'=\mathbf t-\mathbf e_A$.
Every target count remains positive. Every coordinate originally
equal to $r$ has been reduced, and all other coordinates were already
at most $r-1$. Thus
\[
\sum_Xp'_X=2(r-1)+1,
\qquad 1\le p'_X\le r-1.
\]
All induction hypotheses hold. Zero coordinates of the reduced source
cause no problem because the theorem allows arbitrary nonnegative
source counts.

Apply induction to obtain a proper target coloring $y'$ on the
smaller cycle meeting the reduced quotas. Let $L,R$ be its colors at
the endpoints of the new edge. Since $y'$ is proper, $L\ne R$.
Replace that edge by the original two vertices, assigning them
either $AB$ or $BA$. The order $AB$ fails only if $L=A$ or $R=B$;
the order $BA$ fails only if $L=B$ or $R=A$.
If both orders failed, their two conditions, together with $A\ne B$,
would force either $L=R=A$ or $L=R=B$. This is impossible, so at
least one order produces a proper coloring.

The insertion restores the profile $\mathbf p$. Both restored source
vertices were originally $A$, and exactly one has now been colored
$A$. We have therefore added one $A$ agreement, restoring its quota
from $t_A-1$ to $t_A$. All previously retained vertices keep their
colors, so no other quota is lost. This completes the induction.
\end{proof}

\subsection{Directed profile radii}

\begin{definition}[Directed profile radius]\label{odd:directed-definition}
For any nonnegative integer source profile $\mathbf m$ of total $n$
and any feasible target profile $\mathbf p$, put
\[
 R_{C_n}(\mathbf m\to\mathbf p)
 =\max_{\substack{w:V(C_n)\to\mathcal C\\\mathbf m(w)=\mathbf m}}
   \min_{y\in\mathcal P_{\mathbf p}(C_n)}d_{\mathrm H}(w,y).
\]
The source profile need not be feasible.
For a nonempty family $\mathcal Q$ of feasible target profiles, define
$R_{C_n}(\mathbf m\to\mathcal Q)$ by allowing any proper target whose
profile belongs to $\mathcal Q$ in the same maximum--minimum.
\end{definition}

Write $(x)_+=\max\{x,0\}$.
\begin{theorem}[Directed and restricted-target radii]\label{odd:directed}
For every pair in \cref{odd:directed-definition}, with $n=2r+1$,
\begin{align*}
 R_{C_{2r+1}}(\mathbf m\to\mathbf p)
 &=2r+1-\sum_{X\in\mathcal C}
       \min\left\{p_X,\left\lceil\frac{m_X}{2}\right\rceil\right\}\\
 &=\sum_{X\in\mathcal C}\left\lfloor\frac{m_X}{2}\right\rfloor
   +\sum_{X\in\mathcal C}
       \left(\left\lceil\frac{m_X}{2}\right\rceil-p_X\right)_+.
\end{align*}
At most one term in the positive-part sum is nonzero.
More generally, for every nonempty family $\mathcal Q$ of feasible
target profiles,
\begin{equation}\label{odd:target-family}
R_{C_{2r+1}}(\mathbf m\to\mathcal Q)
=2r+1-\max_{\mathbf p\in\mathcal Q}
 \sum_{X\in\mathcal C}\min\{p_X,\lceil m_X/2\rceil\}.
\end{equation}
The block word $A^{m_A}B^{m_B}C^{m_C}$ attains the radius for every
target profile simultaneously.
In particular, every feasible source profile satisfies
\[
 R^{\mathrm w}_{C_{2r+1}}(\mathbf m)
 =R_{C_{2r+1}}(\mathbf m)
 =\sum_{X\in\mathcal C}\left\lfloor\frac{m_X}{2}\right\rfloor.
\]
\end{theorem}
\begin{proof}
The upper bound is \cref{odd:capped}. For the reverse inequality,
use the three-block obstruction of
Clark et al.~\cite[Proposition~3.5]{clark2006}, allowing empty blocks.
If the source is not monochromatic, each nonempty block induces a
path, so a proper target retains at most $\lceil m_X/2\rceil$
vertices of color $X$. Its target profile also limits that number
to $p_X$. Summing gives the lower bound. For a monochromatic
source, every target has exactly $p_X$ agreements in its sole source
color, and the same formula follows directly.
The same block word gives these lower bounds simultaneously for all
$\mathbf p\in\mathcal Q$. For the reverse bound, apply
\cref{odd:capped} with a profile in $\mathcal Q$ maximizing the
displayed agreement sum. This proves \eqref{odd:target-family}.
The second expression is an algebraic rearrangement, and
\eqref{odd:pair-quota} proves the assertion about the positive-part
sum. For a feasible source profile, take $\mathbf p=\mathbf m$
to obtain the ordinary value. The same block obstruction bounds
the weak radius from below, while the ordinary upper bound applies
to the larger family of weak targets. Hence the two profile radii agree.
\end{proof}

\section{Complete multipartite graphs}\label{sec:mp}

On an odd cycle, \cref{odd:directed} determines the repair cost for
each pair of source and target profiles. For a complete multipartite
graph, every optimal proper coloring is constant on each part, so
repair instead amounts to assigning colors to parts. Prescribed
multiplicities become row and column sums of a disturbance matrix.
The transportation criterion in \cref{lem:transport}, which governed
the gap allocations in \cref{odd:interpolation}, will again determine
which integer counts can be realized. The resulting matrix model also
provides the starting point for the extrema in \cref{sec:extremal}.

\subsection{Matrix representation and assignment duality}

For a nonnegative square matrix $M=(m_{ij})$ of order $h$, define its
\emph{assignment score}, or max-plus tropical determinant, by
\[
 \tau(M)=\max_{\pi\in S_h}\sum_{i=1}^h m_{i,\pi(i)}.
\]

Let $G=K_{p_1,\ldots,p_k}=K_{\mathbf p}$, where every $p_i$ is positive, and put
$\mathbf p=(p_1,\ldots,p_k)^{\mathsf T}$ and $n=\sum_i p_i$.
Every proper $k$-coloring of $G$ assigns a distinct color to each part.
We may therefore label the colors so that color $j$ has multiplicity
$p_j$. For a rearrangement $w$, let $m_{ij}$ count the vertices of color
$j$ in part $i$. The possible matrices are precisely
\[
 \mathcal T(\mathbf p)=\{M\in\mathbb Z_{\ge0}^{k\times k}:
 M\mathbf1=\mathbf p,\ M^{\mathsf T}\mathbf1=\mathbf p\}.
\]
We call $M$ the \emph{disturbance matrix} of $w$; its row and column
sum vectors are its \emph{margins}.
Recall that $\tau(M)$ is the maximum assignment score. Assigning color
$\pi(i)$ to part $i$ retains exactly
$\sum_i m_{i,\pi(i)}$ vertices. Consequently,
\begin{equation}\label{eq:mp:matrix-weak}
 B_{\mathrm w}(G)=n-\mu(\mathbf p),\qquad
 \mu(\mathbf p)=\min_{M\in\mathcal T(\mathbf p)}\tau(M).
\end{equation}

For a repair preserving color multiplicities, the assignment is allowed
only when $p_i=p_{\pi(i)}$ for every $i$. Partition $[k]$ into
\emph{size classes}, grouping indices with equal part sizes. If
$M[\Gamma,\Gamma]$ is the principal submatrix belonging to a size class
$\Gamma$, the allowed assignments act independently on these classes.
Thus
\begin{equation}\label{eq:mp:matrix-ordinary}
 B(G)=n-\min_{M\in\mathcal T(\mathbf p)}
       \sum_{\Gamma}\tau(M[\Gamma,\Gamma]).
\end{equation}
The distinction between the two parameters is therefore a distinction
between unrestricted assignments and assignments constrained by the
size classes.

Row and column \emph{potentials} are vectors $\mathbf u,\mathbf v$.
They \emph{cover} $M$ if $u_i+v_j\ge m_{ij}$ for every $i,j$, and their
\emph{cost} is $\sum_i u_i+\sum_jv_j$.
We shall use the following integral form of assignment duality
\cite{kuhn1955}. We include a short proof to fix the normalization
of its nonnegative potentials.

\Needspace{6\baselineskip}
\begin{lemma}[Integral assignment duality]\label{lem:assignment-duality}
For a nonnegative integer matrix $M$ of order $h$,
\[
 \tau(M)=\min\left\{
 \sum_i u_i+\sum_j v_j:
 \mathbf u,\mathbf v\in\mathbb Z_{\ge0}^h,
 \quad u_i+v_j\ge m_{ij}\ \text{for every }i,j
 \right\}.
\]
\end{lemma}
\begin{proof}
Every feasible pair of potentials bounds every assignment score by its
total sum. To obtain equality, choose an optimal permutation and permute
the columns to make its entries diagonal; call the resulting matrix $H$.
Give each directed edge $i\to j$ weight $h_{ij}-h_{ii}$. A directed cycle
of positive weight would replace part of the diagonal assignment by a
larger assignment. Thus every directed cycle has nonpositive weight.

Adjoin a source with a zero-weight edge to every index and let $a_j$ be
the maximum weight of a walk from this source to $j$. Removing a closed
subwalk cannot decrease its weight, so the maximum is attained on a
simple path. In particular, every $a_j$ is a finite integer, and
\[
 a_j\ge a_i+h_{ij}-h_{ii}.
\]
The integer potentials $u_i=h_{ii}-a_i$ and $v_j=a_j$ cover $H$ and have
total cost $\sum_i h_{ii}=\tau(M)$. Restore the original column order to
obtain potentials covering $M$ with the same cost.

These potentials may be signed. Put $c=\min_i u_i$ and replace them by
$u_i-c$ and $v_j+c$. The first vector is nonnegative. A row attaining
$u_i=c$ gives $v_j+c\ge m_{ij}\ge0$ for every $j$, so the second vector
is also nonnegative. The covering inequalities and total cost are unchanged.
\end{proof}

\subsection{Ordinary villainy and the dominant size class}

A size class consisting of $m$ parts of size $d$ contains $md$ vertices.
It is \emph{dominant} when $md>n/2$; at most one such class exists. The next
theorem evaluates the contribution of this class. The size-class
reduction and its maximum-weight matching interpretation already
appear in Raleigh's thesis~\cite[Chapter~II]{raleigh2018}.

We first identify a known extremal case. Let $\mathfrak G$ be the subgroup
of permutations of $V(G)$ that map each part bijectively onto a part of
the same size. Its Hamming covering radius in the full symmetric group
is the largest distance of a vertex permutation from its nearest element
of $\mathfrak G$. This radius equals $B(G)$.
Indeed, fix a proper coloring $c$ constant on the parts.
Every rearrangement has the form $c\circ\pi$, and every count-preserving
proper target has the form $c\circ\gamma$ with $\gamma\in\mathfrak G$.
For a fixed target, retain the partial bijection $\gamma(v)=\pi(v)$
at every color agreement and extend it within the corresponding parts.
This extension has exactly those point agreements. Thus the minimum
permutation distance equals the minimum color distance, and the maxima
also agree. The orbits of $\mathfrak G$ are exactly the unions of all
parts of each fixed size.
Cameron and Wanless~\cite[Theorem~14]{cameronwanless2005} therefore give
the following criterion. Raleigh~\cite[Theorem~1]{raleigh2018} also
proved the full-change conclusion directly when no size class contains
more than half the vertices.

\begin{proposition}[Cameron--Wanless]\label{cor:mp:full-change}
For a complete multipartite graph $G$ on $n$ vertices, $B(G)=n$ if and
only if every size class contains at most $n/2$ vertices.
\end{proposition}

\begin{theorem}\label{thm:mp:ordinary}
If every size class contains at most $n/2$ vertices, then $B(G)=n$.
Otherwise, let $(m,d)$ be the parameters of the unique size class with
$md>n/2$, and put
\[
 P=n-md,\qquad t=md-P,\qquad
 b=d-m\left\lfloor\frac d m\right\rfloor.
\]
If $P\ge mb$, then
\begin{equation}\label{eq:mp:raleigh-regime}
 B(G)=n-\left\lceil\frac t m\right\rceil
 \qquad\text{\cite[Theorem~4]{raleigh2018}}.
\end{equation}
For the remaining case $P<mb$, and for $0\le r,s<m$, set $h=m-r$,
$\ell=m-s$, and
\begin{equation}\label{eq:mp:ordinary-q}
\begin{split}
 q_{r,s}=\max\bigg\{&0,
 \left\lceil\frac{t-m(r+s)}{m^2}\right\rceil,
 \left\lceil\frac{t-dr-hs}{mh}\right\rceil,\\
 &\left\lceil\frac{t-ds-\ell r}{m\ell}\right\rceil,
 \left\lceil\frac{t-d(r+s)}{h\ell}\right\rceil\bigg\}.
\end{split}
\end{equation}
Then
\begin{equation}\label{eq:mp:ordinary}
 B(G)=n-F(m,d,t),\qquad
 F(m,d,t)=\min_{0\le r,s<m}(mq_{r,s}+r+s).
\end{equation}
In particular, the formula uses at most $k^2$ arithmetic candidates.
\end{theorem}

Raleigh~\cite[Theorems~2--3]{raleigh2018} obtained bounds in the
remaining regime $P<mb$. Formula~\eqref{eq:mp:ordinary} evaluates
that regime. The uniform-part problem was determined independently
by Dinitz, Hartman, and Soprunova~\cite{dinitz2012}; its extension
to rectangular transportation matrices is given in~\cite{gajula2015}.
We do not give separate proofs of these previously known formulas.

We isolate the matrix problem behind the remaining case. For integers
$0\le t\le md$, define
\begin{equation}\label{eq:mp:block}
 F(m,d,t)=\min\left\{\tau(H):
 \begin{array}{l}
 H\in\mathbb Z_{\ge0}^{m\times m},\quad
 H\mathbf1\le d\mathbf1,\quad H^{\mathsf T}\mathbf1\le d\mathbf1,\\
 \displaystyle\sum_{i,j}h_{ij}=t
 \end{array}\right\}.
\end{equation}
The two meanings of $F$ in \eqref{eq:mp:ordinary} and
\eqref{eq:mp:block} will be shown to agree.

\begin{lemma}\label{lem:mp:dominant-block}
If a size class with parameters $(m,d)$ has $md>n/2$, then
$B(G)=n-F(m,d,2md-n)$, with $F$ defined by \eqref{eq:mp:block}.
\end{lemma}
\begin{proof}
Write $S=md$, $P=n-S$, and $t=S-P$. In any matrix
$M\in\mathcal T(\mathbf p)$, the block $H$ belonging to this class
has total sum at least $S-P=t$: its rows have total $S$, whereas the
columns outside the class can receive only $P$ units. Its row and column
sums are at most $d$. Delete units from its entries until its total is
$t$. This operation cannot increase an assignment score, so
$\tau(H)\ge F(m,d,t)$. The other diagonal blocks contribute
nonnegative scores to \eqref{eq:mp:matrix-ordinary}.

Conversely, take a matrix $H$ in \eqref{eq:mp:block}. Its row deficits
from $d$ sum to $md-t=P$, exactly the total demand of the outside
columns. Fill the block from the internal rows to the outside columns
with these margins. Similarly, fill the block from the outside rows to
the internal columns using the column deficits of $H$. Both fillings
exist with integer entries: one can repeatedly place the smaller of a
remaining row supply and column demand in their intersection. Set the
outside-to-outside block to zero. The resulting matrix belongs to
$\mathcal T(\mathbf p)$, and $H$ is its only nonzero diagonal size-class
block. Minimizing over $H$ proves the assertion.
\end{proof}

The following elementary balancing fact explains why only two values
are needed in each vector of assignment potentials.

\begin{lemma}\label{lem:mp:balancing}
Among nonnegative integer vectors of a fixed length and sum, a vector
whose entries differ by at most one maximizes the sum of the smallest
$j$ entries, simultaneously for every $j$.
\end{lemma}
\begin{proof}
Suppose a largest entry $a$ and a smallest entry $b$ satisfy
$a\ge b+2$. Transfer one unit from $a$ to $b$. Among the subsets of
$j$ entries with minimum sum after the transfer, choose one that does
not contain the donor without the receiver. Such a choice is possible
because the receiver's new value $b+1$ is at most the donor's new value
$a-1$. On this chosen subset the transfer has not decreased the sum.
Its old sum was at least the old minimum, so the minimum sum of $j$
entries has not decreased. This holds for every $j$. Repeating the
operation terminates, since the sum of squares decreases at each step,
and the final vector has entries differing by at most one.
\end{proof}

\begin{proof}[Proof of \cref{thm:mp:ordinary}]
The case without a dominant size class is
\cref{cor:mp:full-change}, and the case $P\ge mb$ is
Raleigh's theorem quoted in~\eqref{eq:mp:raleigh-regime}.

It remains to treat $md>n/2$ and $P<mb$.
By \cref{lem:mp:dominant-block}, it remains
to evaluate \eqref{eq:mp:block}. For nonnegative integer vectors
$\mathbf u,\mathbf v$ of length $m$, let $\Phi(\mathbf u,\mathbf v)$
be the largest total sum of an integer matrix whose row and column
sums are at most $d$ and whose $(i,j)$ entry is at most $u_i+v_j$.
Integral assignment duality, \cref{lem:assignment-duality}, gives
\begin{equation}\label{eq:mp:block-dual}
 F(m,d,t)=\min\left\{\sum_i u_i+\sum_jv_j:
 \mathbf u,\mathbf v\in\mathbb Z_{\ge0}^m,
 \ \Phi(\mathbf u,\mathbf v)\ge t\right\}.
\end{equation}
Indeed, optimal potentials for a matrix in \eqref{eq:mp:block} give
a feasible pair on the right with the same assignment value.
Conversely, if the capacities admit at least $t$ units, an integral
flow can be reduced to exactly $t$ units. The resulting matrix has
every assignment score bounded by the total potential cost.

In the network defining $\Phi$, source-to-row and column-to-sink
capacities are all $d$. A cut with source-side rows $I$ and sink-side
columns $J$ has capacity
\begin{equation}\label{eq:mp:block-cut}
 d(2m-|I|-|J|)+|J|\sum_{i\in I}u_i+|I|\sum_{j\in J}v_j.
\end{equation}
For fixed $|I|$ and $|J|$, the minimum is obtained by choosing the
smallest potentials on both sides. \cref{lem:mp:balancing}
therefore shows that balancing $\mathbf u$ and $\mathbf v$ separately
cannot decrease any minimum cut and hence cannot decrease $\Phi$.
It preserves the objective in \eqref{eq:mp:block-dual}.

Write the balanced row potentials as $q_1$ and $q_1+1$, with $r$ entries
of the larger value, and the balanced column potentials as $q_2$ and
$q_2+1$, with $s$ entries of the larger value. Here $0\le r,s<m$.
Put
\[
 q=q_1+q_2,\qquad h=m-r,\qquad \ell=m-s,\qquad T=mq+r+s.
\]
The entry capacities are $q,q+1,q+2$, and their total potential cost is
$T$. Every nonnegative integer $q$ is possible, for example by choosing
$q_1=q$ and $q_2=0$.

The sum of the smallest $i$ row potentials is
$iq_1+\max(0,i-h)$; the corresponding column sum is
$jq_2+\max(0,j-\ell)$. Substituting these expressions into
\eqref{eq:mp:block-cut} gives a bilinear function of $i,j$ on each
rectangle separated by the lines $i=h$ and $j=\ell$. On a rectangle,
a bilinear function has a minimum at a corner: first minimize in one
coordinate and then in the other. It is thus enough to test
$i\in\{0,h,m\}$ and $j\in\{0,\ell,m\}$. The resulting flow value is
\begin{equation}\label{eq:mp:balanced-flow}
\begin{split}
 \Phi(q,r,s)=\min\{&md,\ mT,\ dr+(m-r)(T-r),\\
                  &ds+(m-s)(T-s),\ d(r+s)+q(m-r)(m-s)\}.
\end{split}
\end{equation}
Because $t\le md$, the first expression imposes no restriction. The
other four inequalities $\Phi(q,r,s)\ge t$ give, after solving for the
integer $q\ge0$, precisely \eqref{eq:mp:ordinary-q}. Substitution into
\eqref{eq:mp:block-dual} proves \eqref{eq:mp:ordinary}.
\end{proof}

For distinct part sizes, the theorem reduces to the formula
\[
 B(K_{p_1,\ldots,p_k})=\min\{n,2(n-\max_i p_i)\},
\]
proved by Clark et al.~\cite[Proposition~4.1]{clark2006}. We also record
the complete bipartite specialization for later use; it is a special
case of their connected bipartite formula.

\begin{proposition}[Clark et al.~\cite{clark2006}]
\label{cor:mp:bipartite}
If $1\le a<b$, then $B(K_{a,b})=2a$. For $a\ge1$,
$B(K_{a,a})=2\lfloor a/2\rfloor$.
\end{proposition}

\subsection{Weak villainy and ordered potentials}

For weak repairs, every permutation is allowed, and the separate size
classes no longer decouple. Nevertheless, the transportation formulation
gives a finite formula with substantially fewer variables than the
original matrix problem.

Assume $p_1\le\cdots\le p_k$, and define
\begin{equation}\label{eq:mp:shapes}
 \mathcal H(\mathbf p)=\left\{\mathbf h\in\mathbb Z_{\ge0}^{k}:
 \begin{array}{l}
 0=h_1\le\cdots\le h_k,\\
 h_j-h_i\le\max\{1,p_j-p_i\}\quad(1\le i<j\le k)
 \end{array}\right\}.
\end{equation}
We call the elements of $\mathcal H(\mathbf p)$ \emph{potential shapes}.
In particular, $h_i\le p_i$, and entries belonging to one size class
differ by at most one.
For a real vector $\mathbf h$ and $1\le r,s\le k$, let
$A_{r,s}(\mathbf h)$ be the sum of the largest $r$ entries of
$(p_i-sh_i)_{i=1}^k$. In particular, negative entries are retained in
this definition. For $\mathbf h,\boldsymbol\ell\in\mathcal H(\mathbf p)$,
put
\begin{equation}\label{eq:mp:weak-Q}
 Q(\mathbf h,\boldsymbol\ell)=
 \max\left\{0,\max_{1\le r,s\le k}
 \left\lceil\frac{A_{r,s}(\mathbf h)+A_{s,r}(\boldsymbol\ell)-n}{rs}
 \right\rceil\right\}.
\end{equation}

\Needspace{9\baselineskip}
\begin{theorem}\label{thm:mp:weak}
Let $\mathbf p=(p_1,\ldots,p_k)$ with $1\le p_1\le\cdots\le p_k$
and $n=\sum_i p_i$. Then $B_{\mathrm w}(K_{\mathbf p})=n-\mu(\mathbf p)$,
where
\begin{equation}\label{eq:mp:weak}
 \mu(\mathbf p)=
 \min_{\mathbf h,\boldsymbol\ell\in\mathcal H(\mathbf p)}
 \left(kQ(\mathbf h,\boldsymbol\ell)+\sum_i h_i+\sum_i\ell_i\right).
\end{equation}
An attaining pair of vectors determines an extremal rearrangement by
an integral flow construction.
\end{theorem}

The required reduction of the potentials uses the following exchange
argument.

\begin{lemma}\label{lem:mp:ordered-potentials}
Suppose $0\le u_i,v_i\le p_i$ are integers and some matrix in
$\mathcal T(\mathbf p)$ satisfies $m_{ij}\le u_i+v_j$.
Without changing the sum of either vector or losing the existence of
such a matrix, the potentials can be chosen so that, for every $i<j$,
\[
 0\le u_j-u_i\le\max\{1,p_j-p_i\},\qquad
 0\le v_j-v_i\le\max\{1,p_j-p_i\}.
\]
The bounds $0\le u_i,v_i\le p_i$ are preserved.
\end{lemma}
\begin{proof}
By \cref{lem:transport}, the capacities $u_i+v_j$ admit a matrix
in $\mathcal T(\mathbf p)$ exactly when
\begin{equation}\label{eq:mp:additive-cuts}
 p(I)+p(J)-n\le |J|u(I)+|I|v(J)
 \qquad(I,J\subseteq[k]),
\end{equation}
where $p(I)=\sum_{i\in I}p_i$, and similarly for $u(I),v(I)$.
For $|I|=r>0$ and $|J|=s>0$, the largest difference between the two
sides is
$A_{r,s}(\mathbf u)+A_{s,r}(\mathbf v)-n$. The choices of $I$ and
$J$ are independent. Conditions with an empty set hold automatically.

Suppose $i<j$ but $u_i>u_j$, and interchange $u_i,u_j$. For any positive
integer $s$, the old residuals
\[
 a=p_i-su_i,\qquad b=p_j-su_j
\]
satisfy $a\le b$. The two new residuals have the same sum and both lie
in $[a,b]$. Replacing a pair by two numbers of the same sum in their
old interval cannot increase the sum of the largest $r$ entries of the
whole vector. Indeed, a maximizing $r$-subset uses zero, one, or both
entries. Only the middle case needs consideration, and its contribution
is at most the larger old entry. No $A_{r,s}(\mathbf u)$ increases, so
\eqref{eq:mp:additive-cuts} remains valid. The coordinate bounds also
remain valid: the potential moved to $i$ decreases, and the potential
moved to $j$ is at most $p_i\le p_j$. Repeating removes all inversions.

After sorting, suppose that $i<j$ and
$u_j-u_i>\max\{1,p_j-p_i\}$. Put
$\delta=p_j-p_i$ and $\eta=u_j-u_i$, and transfer one unit from
$u_j$ to $u_i$. For every positive integer $s$, the old residuals
$a=p_i-su_i$ and $b=p_j-su_j$ satisfy
\[
 a-b=s\eta-\delta\ge s,
\]
because $\eta\ge\delta+1$. The new residuals are $a-s$ and $b+s$;
both lie in $[b,a]$ and their sum is unchanged. The same contraction
argument therefore preserves every cut inequality.
The coordinate bounds survive as well: $u_i=p_i$ would force
$u_j>p_j$, and $u_j\ge u_i+2$ ensures that the donor stays nonnegative.

Sort the modified vector again, and repeat whenever a pair violates
the asserted bound. Each transfer strictly decreases the sum of squares,
whereas sorting leaves it unchanged, so the procedure terminates.
The sum of the vector and feasibility are preserved at every step.
Apply the procedure first to $\mathbf u$ and then to $\mathbf v$.
\end{proof}

\begin{proof}[Proof of \cref{thm:mp:weak}]
Choose $M\in\mathcal T(\mathbf p)$ attaining $\mu(\mathbf p)$ and
nonnegative integral optimal assignment potentials $\mathbf u,\mathbf v$.
We may assume $u_i\le p_i$ and $v_j\le p_j$. To see this, truncate each
potential at its corresponding margin. If $u_i$ is truncated, its new
value $p_i$ alone is at least every entry in row $i$; the other potential
is nonnegative. The same argument applies to a truncated column
potential. Thus truncation preserves the covering inequalities and
does not increase the cost.

Apply \cref{lem:mp:ordered-potentials}. The modified potentials
need not cover the original matrix, but they cover some matrix with
the same margins, which is sufficient for minimizing
$\mu(\mathbf p)$. Subtract the minimum separately from each vector:
\[
 \alpha=\min_i u_i,\quad \beta=\min_jv_j,\quad q=\alpha+\beta,
 \qquad h_i=u_i-\alpha,\quad \ell_j=v_j-\beta.
\]
Subtracting a constant preserves all difference bounds, so
$\mathbf h,\boldsymbol\ell\in\mathcal H(\mathbf p)$. The
capacities become $q+h_i+\ell_j$, and the cost is
$kq+\sum_i h_i+\sum_j\ell_j$.

For fixed shapes $\mathbf h,\boldsymbol\ell$, substitute these
capacities into \eqref{eq:mp:additive-cuts}. For nonempty $I,J$ of
cardinalities $r,s$, the condition becomes
\[
 rsq\ge p(I)-s h(I)+p(J)-r\ell(J)-n.
\]
Maximizing the right side independently over $I$ and $J$ shows that
all cuts are feasible exactly when $q\ge Q(\mathbf h,\boldsymbol\ell)$.
Decreasing $q$ to this least value therefore preserves feasibility.
This proves that the right side of \eqref{eq:mp:weak} is at most
$\mu(\mathbf p)$.

Conversely, any shapes with $q=Q(\mathbf h,\boldsymbol\ell)$ satisfy
all capacity inequalities. \cref{lem:transport} gives an integer
matrix in $\mathcal T(\mathbf p)$ bounded by these capacities.
Every permutation score is at most
$kq+\sum_i h_i+\sum_j\ell_j$. Hence $\mu(\mathbf p)$ is at most
each candidate in \eqref{eq:mp:weak}, proving equality. For a minimizing
pair, the constructed matrix has assignment value exactly
$\mu(\mathbf p)$ and therefore yields the asserted extremal
rearrangement.
\end{proof}

The difference bounds also control the number of candidates in terms
of the spread of the part sizes. A separate consequence of the
transportation formulation is fixed-parameter tractability in the
number of parts, without a bound on that spread.

\begin{corollary}\label{cor:mp:computation}
Put $\Delta=\max\{1,p_k-p_1\}$. The formula in
\eqref{eq:mp:weak} uses at most
\[
 \binom{k+\Delta-1}{\Delta}^{\!2}
\]
candidate pairs. In particular, at most $k^2$ pairs suffice when the
part sizes differ by at most one.
Moreover, both $B(K_{\mathbf p})$ and $B_{\mathrm w}(K_{\mathbf p})$
can be computed in time $f(k)L^c$, where $L$ is the binary encoding
length of $\mathbf p$, $f$ is computable, and $c$ is an absolute constant.
An extremal disturbance matrix for weak villainy can be constructed
within such a bound.
\end{corollary}
\begin{proof}
Every shape has first entry zero and all entries at most $\Delta$.
There are $\binom{k+\Delta-1}{\Delta}$ nondecreasing choices for its
remaining $k-1$ entries in $\{0,\ldots,\Delta\}$. The defining
difference bounds can only reduce this number. If $p_k-p_1\le1$,
then $\Delta=1$.

For the algorithmic assertion, ordinary villainy is given by
\cref{thm:mp:ordinary}. For weak villainy, minimize
$\sum_i u_i+\sum_jv_j$ over integer potentials satisfying
\[
 0\le u_i,v_i\le p_i\quad(1\le i\le k)
\]
and all the inequalities \eqref{eq:mp:additive-cuts}.
The optimum of this integer program is $\mu(\mathbf p)$.
Indeed, an optimal disturbance matrix has bounded integral assignment
potentials, by the truncation argument in the preceding proof.
Conversely, the transportation criterion produces a matrix with these
margins whose assignment score is at most the potential cost.

The program has $2k$ variables and at most $4^k+4k$ inequalities.
Its encoding length is bounded by a function of $k$ times a polynomial
in $L$. Fixed-dimensional integer programming
\cite{lenstra1983,kannan1983}, with binary search on the integer
objective in $[0,n]$, therefore gives the asserted running time.
From optimal potentials, a polynomial-time integral maximum-flow
algorithm constructs an attaining matrix. This output records the
disturbance by its $k^2$ entries; expanding it into a color on every
vertex is unnecessary.
\end{proof}

If all part sizes are equal, the matrix problem is the uniform-margin
problem solved in~\cite{dinitz2012,gajula2015}. The present formulation
permits arbitrary positive margins.

\section{Extremal villainy at fixed order and chromatic number}
\label{sec:extremal}

For integers $1\le k\le n$, define
\begin{align*}
  \mathcal M(n,k)&=\max\{B(G): |V(G)|=n,\ \chi(G)=k\},\\
  \mathcal M_{\mathrm w}(n,k)
    &=\max\{B_{\mathrm w}(G): |V(G)|=n,\ \chi(G)=k\}.
\end{align*}
We now use the multipartite description in \cref{sec:mp} to determine
both maxima over all graphs of prescribed order and chromatic number.
\Cref{ext:completion} reduces the problem to complete multipartite
graphs. The ordinary extremum then uses the size-class criterion in
\cref{cor:mp:full-change}; the weak extremum uses the assignment model
\eqref{eq:mp:matrix-weak} and the integral duality of
\cref{lem:assignment-duality}. In the weak problem, balance of the
disturbance matrix imposes a further parity restriction.

The completion must preserve a proper coloring from which the
disturbance arose: arbitrary edge addition need not preserve the set
of legal disturbances.

\Needspace{6\baselineskip}
\begin{lemma}[Completion relative to a source coloring]
\label{ext:completion}
For $F\in\{B,B_{\mathrm w}\}$,
\[
  \max_{\substack{|V(G)|=n\\ \chi(G)=k}}F(G)
   =\max_{\substack{p_1+\cdots+p_k=n\\p_i\ge1}}
      F(K_{p_1,\ldots,p_k}).
\]
\end{lemma}

\begin{proof}
Fix a graph $G$ with $n$ vertices and chromatic number $k$.  Choose a
legal disturbance $w$ attaining $F(G)$, and choose a proper $k$-coloring
$f$ whose labeled color counts agree with those of $w$.  Write
$V_1,\ldots,V_k$ for the color classes of $f$.  They are nonempty,
because $\chi(G)=k$.

Add all missing edges between distinct classes, obtaining
$H=K_{|V_1|,\ldots,|V_k|}$.  The coloring $f$ remains proper on $H$, so
$w$ is still a legal disturbance.  Every proper $k$-coloring of $H$ is
proper on $G$.  This inclusion also holds after restricting to colorings
with the same labeled counts as $w$.  Thus the minimum repair cost for
$w$ cannot decrease, under either repair rule, and $F(G)\le F(H)$.
The reverse inequality between the two maxima follows because the
complete $k$-partite graphs are among the graphs on the left.
\end{proof}

\subsection{Count-preserving repairs}

\begin{theorem}
\label{ext:ordinary}
For $1\le k\le n$, one has $\mathcal M(n,1)=0$ and
\[
  \mathcal M(n,2)=2\left\lfloor\frac{n-1}{2}\right\rfloor
  \qquad(n\ge2).
\]
For $k\ge3$, put $d=n-k$.  Then
\[
  \mathcal M(n,k)=
  \begin{cases}
    n-1,& n=4d-1\ \text{or}\ (n,k)\in\{(5,3),(7,3)\},\\[2pt]
    \min\{n,4d\},&\text{otherwise}.
  \end{cases}
\]
\end{theorem}

The proof uses the criterion in \cref{cor:mp:full-change}:
if $\mathbf p=(p_1,\ldots,p_k)$ has total $n$, then
\begin{equation}
\label{ext:full-change}
 B(K_{\mathbf p})=n
 \quad\Longleftrightarrow\quad
 s\,|\{i:p_i=s\}|\le\frac n2
 \quad\text{for every occurring part size }s.
\end{equation}
The relevant quantity is therefore the combined size of all parts of
the same size, rather than the size of an individual part.

\begin{proof}[Proof of \cref{ext:ordinary}]
We first establish the upper bound
\begin{equation}
\label{ext:ordinary-bound}
 \mathcal M(n,k)\le\min\{n,4(n-k)\}.
\end{equation}
By \cref{ext:completion}, it is enough to consider
$K_{\mathbf p}$.  Let $a$ be its number of nonsingleton parts, and let
$T$ be their total size.  Since
\[
 d=n-k=\sum_{i=1}^k(p_i-1),
\]
each nonsingleton part contributes at least one to $d$, whence
$a\le d$ and $T=d+a\le2d$.

Represent a disturbance by its matrix $M=(m_{ij})$, with row $i$
corresponding to part $i$ and column $j$ to color $j$.  As in
\eqref{eq:mp:matrix-ordinary}, the corresponding row and column sums
are $p_i$.  The rows belonging to singleton parts have total mass
$n-T$.  At most $T$ of this mass can lie in nonsingleton columns.
Consequently the singleton-to-singleton block has mass at least
$n-2T\ge n-4d$.  Each positive entry in this block is $1$, and no two
positive entries share a row or a column.  These entries extend to a
permutation of the singleton indices.  Completing the permutation
within each of the other part-size classes gives an allowed repair
retaining all those entries.  Thus at most $4d$ vertices need to change,
which proves \eqref{ext:ordinary-bound}.

Suppose first that $n\ge4d$.  Take $d$ parts of size $2$ and $n-2d$
singleton parts; their number is $n-d=k$.  For each size-$2$ part,
choose two distinct, previously unused singleton parts.  Move the two
tokens of its source color to these singletons, and move the two
singleton colors into the size-$2$ part.  There are enough singleton
parts because $n-2d\ge2d$.  Leave the remaining singleton colors fixed.
The $4d$ exchanged vertices carry colors whose multiplicities differ
from the sizes of their parts, so none can be retained by a
count-preserving repair.  All remaining vertices can be retained.
This gives equality in \eqref{ext:ordinary-bound}.  The construction
also includes $d=0$, when no exchanges are made.

Next suppose $n=4d-1$, so $k=3d-1$.  In any part-size vector, if
$a<d$, the singleton class has total size
\[
 k-a\ge (3d-1)-(d-1)=2d>n/2.
\]
If $a=d$, then every nonsingleton part has size $2$, and this part-size
class has total size $2d>n/2$.  In either case,
\eqref{ext:full-change} excludes villainy $n$, so
$\mathcal M(n,k)\le n-1$.  To attain $n-1$, take $d$ size-$2$ parts
and $2d-1$ singletons.  Perform the preceding exchanges for $d-1$ of
the size-$2$ parts.  In the last size-$2$ part, exchange one color
token with the last singleton and leave the other token fixed.
Exactly one unit of the disturbance matrix lies in blocks whose row
and column parts have the same size.  It is a diagonal entry, so the
maximum number of retained vertices is exactly one.

It remains to find profiles satisfying \eqref{ext:full-change} when
$n<4d-1$, apart from the two stated small exceptions.  The following
construction also covers the boundary $n=4d$.  Put
\[
 q=4d-n=3n-4k.
\]
We claim that such a profile exists whenever $k\ge3$, $q\ge0$, and
$q\ne1$, except for $(n,k)=(5,3)$ and $(7,3)$.

If a profile satisfies \eqref{ext:full-change}, adjoining the three
parts $(1,1,2)$ preserves the criterion.  The order increases by $4$,
so the permitted mass $n/2$ increases by $2$.  The size-$1$ and
size-$2$ classes each gain exactly $2$, and all other classes are
unchanged.  It is therefore enough to use the seed profiles in the
following table and repeatedly adjoin $(1,1,2)$.
\begin{center}
\begin{tabular}{ccl}
Number of parts & Order & Part sizes\\
\hline
$3$ & $4$ & $(1,1,2)$\\
$3$ & $2m\ge6$ & $(1,m-1,m)$\\
$3$ & $2m+1\ge9$ & $(2,m-1,m)$\\
$4$ & $2m\ge6$ & $(1,1,m-2,m)$\\
$4$ & $2m+1\ge7$ & $(1,1,m-1,m)$\\
$5$ & $2m\ge8$ & $(1,1,1,m-3,m)$\\
$5$ & $2m+1\ge9$ & $(1,1,1,m-2,m)$\\
$6$ & $9$ & $(1,1,1,1,2,3)$\\
$6$ & $11$ & $(1,1,1,1,2,5)$
\end{tabular}
\end{center}
Each seed satisfies \eqref{ext:full-change}.  In the ranged families,
the nonsingleton sizes are distinct and at most half the order.  The
only endpoints at which an additional displayed part becomes a
singleton are $(1,1,1,3)$ and $(1,1,1,1,4)$; their singleton masses
are also at most half the order.  The largest equal-size class in the
two six-part seeds has mass $4$ and $5$, respectively.

For completeness, these seeds cover all the claimed parameters.
Choose $k_0\in\{3,4,5\}$ with $k_0\equiv k\pmod3$, and set
\[
 j=\frac{k-k_0}{3},\qquad n_0=n-4j.
\]
Then $j\ge0$ and $3n_0-4k_0=q$.  For $k_0=4$, this gives
$n_0\ge6$, so a four-part seed applies.  For $k_0=5$, it gives
$n_0\ge7$; equality would imply $q=1$, which is excluded, so a
five-part seed applies.  For $k_0=3$, it gives $n_0\ge4$, and a
three-part seed applies unless $n_0\in\{5,7\}$.  In those two cases,
if $j\ge1$, use the six-part seed of order $n_0+4$ and adjoin
$j-1$ copies of $(1,1,2)$.  If $j=0$, the parameters are precisely
the two excluded pairs.  In every other case, adjoin $j$ copies to
the chosen $k_0$-part seed.  This proves the claim, and hence yields
$\mathcal M(n,k)=n$ throughout this range.

For $(n,k)=(5,3)$, the possible unordered part-size vectors are
$(1,1,3)$ and $(1,2,2)$.  For $(n,k)=(7,3)$, they are
\[
 (1,1,5),\quad (1,2,4),\quad (1,3,3),\quad (2,2,3).
\]
Every one has an equal-size class of mass greater than $n/2$, so the
extremal villainy is at most $n-1$.  The disturbance matrices
\[
 \begin{pmatrix}0&0&1\\0&0&1\\1&1&1\end{pmatrix},
 \qquad
 \begin{pmatrix}0&0&1\\0&0&2\\1&2&1\end{pmatrix}
\]
have respective margin vectors $(1,1,3)$ and $(1,2,4)$.  In each
matrix, an allowed permutation retains at most one vertex, and the
identity retains one.  Both upper bounds are therefore attained.

Finally, $k=1$ gives zero villainy.  For $k=2$, apply the complete
bipartite formulas in \cref{cor:mp:bipartite}.  Unequal
parts $a<b$ give $B(K_{a,b})=2a$, whereas equal parts give
$B(K_{a,a})=2\lfloor a/2\rfloor$.  Their maximum at fixed sum $n$
is $2\lfloor(n-1)/2\rfloor$.  It is attained by
$((n-1)/2,(n+1)/2)$ for odd $n$, and by
$(n/2-1,n/2+1)$ for even $n\ge4$.  When $n=2$, the graph is $K_2$
and the value is zero.
\end{proof}

\subsection{Weak repairs and an assignment parity obstruction}

Call a square matrix \emph{balanced} if its $i$th row sum equals its
$i$th column sum for every $i$.  The row sums need not all be equal.
Let $\mathcal B_{n,k}$ be the set of balanced, nonnegative integer
$k\times k$ matrices with total sum $n$ and with all row sums positive.
\cref{ext:completion} and \eqref{eq:mp:matrix-weak} give
\begin{equation}
\label{ext:weak-matrix}
 \mathcal M_{\mathrm w}(n,k)
   =n-\min_{M\in\mathcal B_{n,k}}\tau(M),
 \qquad
 \tau(M)=\max_{\sigma\in S_k}\sum_{i=1}^k m_{i,\sigma(i)}.
\end{equation}
\begin{theorem}
\label{ext:weak}
One has $\mathcal M_{\mathrm w}(n,1)=0$.  For $2\le k\le n$, put
\[
 L=\max\left\{\left\lceil\frac nk\right\rceil,\,2k-n\right\},
\]
and define
\[
 \varepsilon=
 \begin{cases}
  1,& L\text{ is odd},\quad (k\text{ is even or }L<k),
       \quad n\ge k(L-1)+2,\\
  0,&\text{otherwise}.
 \end{cases}
\]
Then
\[
 \mathcal M_{\mathrm w}(n,k)=n-L-\varepsilon.
\]
The maximum is attained by a complete $k$-partite graph with at most
three distinct part sizes.
\end{theorem}

We prove the theorem after two assignment bounds. Besides averaging,
integrality and balance impose the parity restriction responsible for
$\varepsilon$.

\begin{lemma}
\label{ext:elementary-assignment}
If $M$ is a nonnegative integer $k\times k$ matrix of total sum $n$
with positive row and column sums, then
\[
 \tau(M)\ge
 \max\left\{\left\lceil\frac nk\right\rceil,\,2k-n\right\}.
\]
\end{lemma}

\begin{proof}
Each entry occurs in exactly $(k-1)!$ of the $k!$ permutation sums.
Their average is consequently $n/k$, giving the first bound.

For the second, form the bipartite graph with one vertex for each row
and each column, joining a row to a column when their entry is positive.
This graph has $2k$ vertices, no isolated vertices, and at most $n$
edges.  If it has $c$ components and $e$ edges, then
$e\ge2k-c$, since each connected component on $v$ vertices has at
least $v-1$ edges.  Hence $c\ge2k-n$.  Choose one edge from each
component.  These edges form a matching, and their entries have total
weight at least $c$.  Extend this matching to a permutation, using
zero entries when necessary.  Nonnegativity then gives
$\tau(M)\ge c\ge2k-n$.
\end{proof}

The elementary bounds use only positivity of the margins. Balance
adds a restriction when the assignment score is close to its average
lower bound: the two sets of optimal potentials must then differ by
a constant. This is the source of the parity correction.

\Needspace{6\baselineskip}
\begin{lemma}[Balanced assignment bound]
\label{ext:odd-assignment}
Let $M$ be a balanced, nonnegative integer matrix of order $k\ge2$,
with total sum $n$ and assignment score $t=\tau(M)$. If
\[
 n>k(t-1)+1,
\]
then there exist $\mathbf a\in\mathbb Z_{\ge0}^{k}$ and
$c\in\{0,1\}$ such that
\begin{equation}\label{ext:additive-cover}
 m_{ij}\le a_i+a_j+c\quad(i,j\in[k]),
 \qquad t=2\sum_i a_i+kc.
\end{equation}
Consequently, if $t$ is odd and either $k$ is even or $t<k$, then
\[
 n\le k(t-1)+1.
\]
\end{lemma}

\begin{proof}
Take nonnegative integer optimal potentials $\mathbf u,\mathbf v$
from \cref{lem:assignment-duality}, retaining the original column
labels. Thus $m_{ij}\le u_i+v_j$ and $\sum_i(u_i+v_i)=t$.
Set
\[
 z_{ij}=u_i+v_j-m_{ij},\qquad
 d_i=u_i-v_i,\qquad \eta=\sum_i d_i.
\]
Nonnegativity and balance give
\begin{equation}
\label{ext:residual-balance}
 \Delta:=\sum_{i,j}z_{ij}=kt-n,
 \qquad
 \sum_jz_{ij}-\sum_jz_{ji}=kd_i-\eta.
\end{equation}

Suppose that the integers $d_i$ are not all equal. Let
$b=\min_i d_i$, $I=\{i:d_i=b\}$, $q=|I|$, and
$S=\sum_i(d_i-b)$. Then $1\le q\le k-1$ and $S\ge k-q$.
Since $\eta=kb+S$, summing the second identity
in~\eqref{ext:residual-balance} over $I$ and cancelling internal
terms gives
\[
 \sum_{\substack{i\in I\\j\notin I}}z_{ji}
 -\sum_{\substack{i\in I\\j\notin I}}z_{ij}=qS.
\]
It follows that
\[
 \Delta\ge qS\ge q(k-q)\ge k-1,
\]
where the last inequality is equivalent to
$(q-1)(k-q-1)\ge0$.

Under the stated hypothesis, $\Delta<k-1$, so all $d_i$ equal an
integer $d$. Put
\[
 a_i=\min\{u_i,v_i\}+\left\lfloor\frac{|d|}{2}\right\rfloor,
 \qquad c=|d|-2\left\lfloor\frac{|d|}{2}\right\rfloor.
\]
These quantities are nonnegative integers and $c\in\{0,1\}$.
Because $u_i-v_i=d$ for every $i$, we have
\[
 u_i+v_j=a_i+a_j+c.
\]
Summing the corresponding diagonal identity gives
$t=2\sum_i a_i+kc$, proving~\eqref{ext:additive-cover}.
If $t$ is odd, this identity forces $c=1$ and $k$ odd, and hence
$t\ge k$. Its contrapositive gives the final assertion.
\end{proof}

\begin{proof}[Proof of \cref{ext:weak}]
The case $k=1$ is immediate.  Assume $k\ge2$ and write
$t=L+\varepsilon$.  \cref{ext:elementary-assignment} gives
$\tau(M)\ge L$ for every $M\in\mathcal B_{n,k}$.  If
$\varepsilon=1$, \cref{ext:odd-assignment} rules out
$\tau(M)=L$.  Thus
\begin{equation}
\label{ext:target-assignment}
 \tau(M)\ge t\qquad(M\in\mathcal B_{n,k}).
\end{equation}
We construct a matrix attaining equality.  The constructions are
symmetric, which ensures balance automatically.

First suppose $k\le n\le2k-2$, and put $r=n-k$.  On $r+2$ indices,
take the adjacency matrix of a star with $r+1$ leaves.  On each of
the remaining $k-r-2$ indices, put a diagonal entry $1$; set every
other entry to zero.  The total sum is
\[
 2(r+1)+(k-r-2)=n,
\]
and every row sum is positive.  A permutation can use at most two
positive entries from the star: at most one lies in the center row
and at most one in the center column.  Two are attainable by
exchanging the center with a leaf.  Fixing all the remaining indices
therefore gives
\[
 \tau(M)=2+(k-r-2)=2k-n=L.
\]
Here $\varepsilon=0$: if $L$ is odd, then $L\ge3$, and
$n\le2k-2<k(L-1)+2$.  Hence this matrix has the required weight $t$.

For the remainder of the proof, assume $n\ge2k-1$.  Then
$L=\lceil n/k\rceil\ge2$.

Suppose next that $t=2h$ is even.  Use index $1$ as the center of a
weighted star and set
\[
 m_{1i}=m_{i1}=a_i\quad(2\le i\le k),\qquad m_{11}=b,
\]
where $1\le a_i\le h$ and $0\le b\le2h$; all other entries are zero.
The sum $S=\sum_{i=2}^k a_i$ can be any integer between $k-1$ and
$h(k-1)$.  For each such $S$, the possible total sums $2S+b$ fill the
integer interval $[2S,2S+2h]$.  Consecutive intervals overlap, because
$h\ge1$.  Together they cover every integer from $2k-2$ to $2hk$.
As $2k-1\le n\le kt=2hk$, the entries can be chosen to give total
sum $n$. For the chosen $S$, take the leaf weights as equally as
possible, so that each is either $\lfloor S/(k-1)\rfloor$ or
$\lceil S/(k-1)\rceil$; both values that occur lie between $1$ and $h$.
Every row sum is positive.  The potentials
\[
 u_1=v_1=h,\qquad u_i=v_i=0\quad(i>1)
\]
satisfy $m_{ij}\le u_i+v_j$, so every permutation has weight at most
$2h=t$.  Equality follows from \eqref{ext:target-assignment}.

It remains to consider odd $t$.  The definition of $\varepsilon$
then forces $\varepsilon=0$ and $t=L$.  If $k$ is even or $t<k$,
the fact that $\varepsilon=0$ implies $n\le k(t-1)+1$.  On the
other hand, $t=\lceil n/k\rceil$ implies $n>k(t-1)$.  Thus
\[
 n=k(t-1)+1.
\]
Write $t=2h+1$, with $h\ge1$.  In the weighted-star construction,
take every $a_i=h$ and take $b=t$.  Its total sum is
\[
 2h(k-1)+(2h+1)=k(t-1)+1=n.
\]
The identity permutation has weight $t$.  The potentials
$u_1=h+1$, $v_1=h$, with all other potentials zero, give an upper
bound of $t$, so $\tau(M)=t$.

Finally, suppose both $t$ and $k$ are odd and $t\ge k$. In
particular, $k\ge3$. The additive form in
\cref{ext:odd-assignment} suggests starting with a matrix of total
sum $kt$ whose assignment weights all equal $t$. Set
\[
 a_1=\frac{t-k}{2},\qquad a_i=0\quad(i>1),
 \qquad m_{ij}=1+a_i+a_j.
\]
This symmetric matrix has total sum $kt$, and every permutation has
weight
\[
 k+2\sum_i a_i=t.
\]
Put $\delta=kt-n$.  Since $t=\lceil n/k\rceil$, we have
$0\le\delta\le k-1$.  Choose $\delta$ indices from
$\{2,\ldots,k\}$ and reduce their diagonal entries from $1$ to $0$.
The resulting matrix remains symmetric, has positive row sums, and
has total sum $n$.  No assignment weight increases.  A cyclic
permutation of the $k$ indices avoids every altered diagonal entry
and still has weight $t$.  This gives the desired matrix in the last
case.

Each constructed matrix belongs to $\mathcal B_{n,k}$ and hence
represents a legal disturbance on the complete multipartite graph
whose part sizes are its row sums.  Equation~\eqref{ext:weak-matrix}
and \eqref{ext:target-assignment} now prove the formula.

These extremal graphs need at most three distinct part sizes.
The first and third constructions have at most two distinct row sums.
In the even weighted-star construction, the leaf weights take at most
two consecutive values, and the center supplies at most one further
row sum. In the final construction, writing $a=(t-k)/2$, the row sums
are $k+(k+1)a$ at the first index and either $k+a$ or $k+a-1$ at the
other indices. Thus there are at most three row sums in every case.
\end{proof}

\section{Concluding remarks}\label{sec:questions}

Prescribed multiplicities and agreement quotas can be met simultaneously
on odd cycles, while complete multipartite repairs reduce to integral
assignments. Optimizing the latter description over the part sizes
determines the extrema over all graphs of fixed order and chromatic
number. The remaining questions concern which parts of these
descriptions extend beyond the settings treated here.

\begin{problem}\label{problem:profile-interpolation}
Beyond the odd-cycle case of \cref{odd:interpolation}, characterize
graph--palette pairs for which every extendable nonempty
partial coloring admits every integer profile coordinatewise between
two of its extension profiles, subject to the fixed total.
For graphs of bounded treewidth and a fixed palette, determine sharp
simultaneous per-color agreement quotas when a feasible target profile
is prescribed.
\end{problem}

The second question concerns guaranteed agreement counts, in addition
to existence of a coloring with the prescribed profile. For comparison,
Bradshaw, Masa\v{r}\'ik, and Stacho~\cite{bradshaw2022} prove that
graphs of treewidth two are $1/3$-flexibly three-choosable, without
prescribing the target multiplicities. This means that, for arbitrary
three-element lists of allowed colors and a request for one listed
color at each vertex of any subset $R$, some proper coloring from the
lists satisfies at least $|R|/3$ requests.

The ordinary multipartite formula has at most $k^2$ candidates.
For weak villainy, the ordered-potential search depends on the
spread of the part sizes, and \cref{cor:mp:computation} gives a
polynomial-time algorithm for every fixed number of parts.

\begin{question}\label{question:weak-arithmetic}
Can $B_{\mathrm w}(K_{p_1,\ldots,p_k})$ be computed in time
polynomial in the binary encoding length of $(p_1,\ldots,p_k)$,
when the number of parts is also part of the input?
\end{question}

\Cref{ext:completion} guarantees a complete multipartite extremizer,
but need not capture all graphs attaining the same value.

\begin{problem}\label{problem:extremal-structure}
Characterize the graphs attaining $\mathcal M(n,k)$ and
$\mathcal M_{\mathrm w}(n,k)$. In particular, determine which edges
can be removed from a complete multipartite extremizer while preserving
both its chromatic number and its extremal villainy.
\end{problem}

One can also let $B_s(G)$ denote the worst repair cost when the target
profile may differ from the source profile by at most $2s$ in
$\ell_1$ distance. Then $B_0(G)=B(G)$, and $B_s(G)=B_{\mathrm w}(G)$
as soon as
\[
 s\ge\min\{B_{\mathrm w}(G),\,|V(G)|-\chi(G)\}.
\]
Indeed, changing $d$ vertices changes the profile by at most $2d$,
and two positive $k$-color profiles on $n$ vertices have
$\ell_1$ distance at most $2(n-k)$.
For odd cycles, \cref{odd:directed} already gives equality for every
$s\ge0$, even with the source profile fixed. When
$B(G)>B_{\mathrm w}(G)$, determining the intermediate values and the
least profile freedom needed to attain the weak value remains a
further direction.

\section*{Acknowledgments}

Lujia Wang is partially supported by the National Natural Science
Foundation of China (grant No.~12371359). Jiang Jiang, Yuhan Jiang,
Yiyang Li, and Aitao Zhang are
supported by the Math+AI research project of Hainan Bielefeld University
of Applied Sciences.

\section*{Contributions and use of AI}

The authors used AI tools from OpenAI to assist with developing and
checking arguments, computational experiments, literature searches,
and manuscript preparation.
The authors are responsible for all proofs, computations, and conclusions.



\begin{thebibliography}{99}
\bibitem{Birken2026}
M.~Birken,
A Hajnal--Szemer\'edi theorem for skewed colorings,
arXiv:2609.18629 (2026).
\url{https://arxiv.org/abs/2609.18629}.

\bibitem{Bowdoin2025}
M.~C.~Bowdoin, Y.~Chi, C.~B.~Ellington, B.~Ives, S.~Lee,
F.~Morrissette, and J.~A.~Mudrock,
An Ohba-like result for flexible list coloring,
arXiv:2509.24013 (2025), revised 2026.
\url{https://arxiv.org/abs/2509.24013}.

\bibitem{bradshaw2022}
P.~Bradshaw, T.~Masa\v{r}\'ik, and L.~Stacho,
Flexible list colorings in graphs with special degeneracy conditions,
\emph{J. Graph Theory} \textbf{101} (2022), no.~4, 717--745.
\url{https://doi.org/10.1002/jgt.22849}.

\bibitem{cameronwanless2005}
P.~J.~Cameron and I.~M.~Wanless,
Covering radius for sets of permutations,
\emph{Discrete Math.} \textbf{293} (2005), no.~1--3, 91--109.
\url{https://doi.org/10.1016/j.disc.2004.08.024}.

\bibitem{clark2006}
S.~A.~Clark, J.~E.~Holliday, S.~H.~Holliday, P.~D.~Johnson, Jr.,
J.~E.~Trimm, R.~R.~Rubalcaba, and M.~Walsh,
Chromatic villainy in graphs,
\emph{Congr. Numer.} \textbf{182} (2006), 171--182.

\bibitem{dinitz2012}
T.~Dinitz, M.~Hartman, and J.~Soprunova,
Tropical determinant of integer doubly-stochastic matrices,
\emph{Linear Algebra Appl.} \textbf{436} (2012), no.~5, 1212--1227.
\url{https://doi.org/10.1016/j.laa.2011.07.033}.

\bibitem{dvorak2019}
Z.~Dvo\v{r}\'ak, S.~Norin, and L.~Postle,
List coloring with requests,
\emph{J. Graph Theory} \textbf{92} (2019), no.~3, 191--206.
\url{https://doi.org/10.1002/jgt.22447}.

\bibitem{fordfulkerson1956}
L.~R.~Ford, Jr., and D.~R.~Fulkerson,
Maximal flow through a network,
\emph{Canad. J. Math.} \textbf{8} (1956), 399--404.
\url{https://doi.org/10.4153/CJM-1956-045-5}.

\bibitem{gajula2015}
S.~Gajula, I.~Soprunov, and J.~Soprunova,
Tropical determinant on transportation polytopes,
\emph{Linear Algebra Appl.} \textbf{475} (2015), 28--44.
\url{https://doi.org/10.1016/j.laa.2015.01.036}.

\bibitem{jahanbekam2021}
S.~Jahanbekam and M.-R.~Lin,
Characterization of graphs with villainy 2,
arXiv:2103.05816 (2021).
\url{https://arxiv.org/abs/2103.05816}.

\bibitem{kannan1983}
R.~Kannan,
Improved algorithms for integer programming and related lattice problems,
in \emph{Proceedings of the Fifteenth Annual ACM Symposium on
Theory of Computing (STOC 1983)}, 193--206.
\url{https://doi.org/10.1145/800061.808749}.

\bibitem{Kaul2024}
H.~Kaul, R.~Mathew, J.~A.~Mudrock, and M.~J.~Pelsmajer,
Flexible list colorings: Maximizing the number of requests satisfied,
\emph{J. Graph Theory} \textbf{106} (2024), no.~4, 887--906.
\url{https://doi.org/10.1002/jgt.23103}.

\bibitem{Kuchukova2026}
A.~Kuchukova, W.~Perkins, and X.~Povill,
Sampling colorings with fixed color class sizes,
in \emph{53rd International Colloquium on Automata, Languages, and Programming
(ICALP 2026)}, LIPIcs \textbf{374} (2026), 134:1--134:22.
\url{https://doi.org/10.4230/LIPIcs.ICALP.2026.134}.

\bibitem{kuhn1955}
H.~W.~Kuhn,
The Hungarian method for the assignment problem,
\emph{Naval Res. Logist. Quart.} \textbf{2} (1955), 83--97.
\url{https://doi.org/10.1002/nav.3800020109}.

\bibitem{lenstra1983}
H.~W.~Lenstra, Jr.,
Integer programming with a fixed number of variables,
\emph{Math. Oper. Res.} \textbf{8} (1983), no.~4, 538--548.
\url{https://doi.org/10.1287/moor.8.4.538}.

\bibitem{raleigh2018}
A.~Raleigh,
The chromatic villainy of complete multipartite graphs,
M.S. thesis, Rochester Institute of Technology, 2018.
\url{https://repository.rit.edu/theses/9889/}.

\bibitem{west2006}
D.~B.~West,
Chromatic villainy of graphs,
online problem collection, problem dated 2006, presented at REGS 2013.
\url{https://dwest.web.illinois.edu/regs/chromvil.html}
(accessed 23 September 2026).

\end{thebibliography}
\end{document}